\documentclass[reqno]{amsart}
\usepackage{willie_style}

\usepackage{xr-hyper}
\title{Skew Howe duality for $\Uqgln$ via quantized Clifford algebras}
\author[W.~Aboumrad]{Willie Aboumrad}
\address[W.~Aboumrad]{The Institute for Computational and Mathematical Engineering (ICME) at Stanford University}
\email{willieab@stanford.edu}
\urladdr{https://web.stanford.edu/~willieab}


\date{}

\begin{document}
\maketitle

\begin{abstract}
	We develop an operator commutant version of the First Fundamental Theorem of invariant theory for the general linear quantum group $\Uqgln$ by using a double centralizer property inside a quantized Clifford algebra. In particular, we show that $\Uqglm$ generates the centralizer of the $\Uqgln$-action on the tensor product of braided exterior algebras $\bigwedge_q(\mathbb{C}^n)^{\otimes m}$. We obtain a multiplicity-free decomposition of the $\Uqgln \otimes \Uqglm$-module $\bigwedge_q(\mathbb{C}^n)^{\otimes m} \cong \bigwedge_q(\mathbb{C}^{nm})$ by computing explicit joint highest weight vectors. We find that the irreducible modules in this decomposition are parametrized by the same dominant weights as in the classical case of the well-known skew $GL_n \times GL_m$-duality. Clifford algebras are an essential feature of our work: they provide a unifying framework for classical and quantized skew Howe duality results that can be extended to include orthogonal algebras of types $\mathbf{BD}$. 
\end{abstract}

\section{Introduction}
In this \chorpaper\ we develop a classical and a quantum \textit{skew} Howe duality result for Type $\mathbf{A}$ algebras. The classical result lays the foundations for our quantized theorem, and the quantized result suggests the generalization to orthogonal types obtained in \crossrefsonch. Clifford algebras are an essential feature of our work, as they provide a unifying framework for discussing duality results for skew-symmetric variables in both the classical and quantum cases, for Types $\mathbf{ABD}$.

Our duality results extend a long and rich tradition in invariant theory tracing back to at least Weyl's prominent book, \textit{The Classical Groups} \cite{weyl_1939}. The fundamental theorems of classical invariant theory may be formulated as a solution to the following problem: given a module $V$ for a reductive group $G$ over a field $k$, provide a complete description of $\End_{k[G]}(V^{\otimes m})$. The First Fundamental Theorem (FFT) gives generators for the centralizer algebra and the Second Fundamental Theorem (SFT) describes all relations amongst them \cite{GW,weyl_1939}. A typical case addressed by the fundamental theorems is when $G$ is one of the classical complex Lie groups and $V$ is the natural $G$-module. In these cases, the fundamental theorems are known as Schur-Weyl-Brauer duality. For details, see Section~5.6 in \cite{GW}. 

Various results are known in the quantized setting. For instance, when $\lieg$ is a classical complex Lie algebra and $V$ is the natural module of the quantized enveloping algebra $\Uqg$, $\End_{\Uqg}(V^{\otimes m})$ is a quotient of a Hecke algebra if $\lieg = \gln$ and it is a quotient of a Birman-Murakami-Wenzl (BMW) algebra in the remaining cases \cite{jimbo_1986,leduc_ram_1997,LZ06}. Lehrer and Zhang describe $\End_{\Uqg}(V^{\otimes m})$ for any irreducible $U_q(\mathfrak{sl}_2)$-module $V$ and when $V$ is the $7$-dimensional irreducible $U_q(\lieg_2)$-module in \cite{LZ06}. Wenzl covered some cases where $\lieg$ is of Lie types $\mathbf{D}$ and $\mathbf{E}$ in \cite{wenzl_spin_centralizer,wenzl_2003}. 

There are various equivalent formulations of the fundamental theorems; in this \chorpaper\ and the next we consider \textit{operator commutant} versions \cite[Theorems~4.2.5,~4.3.4]{howe1995}. We adopt Howe's perspective, which uses multiplicity-free actions satisfying a double centralizer property as an organizing principle, and apply it to the setting of $q$-skew-symmetric variables. 

In this \chorpaper\, we focus Type $\mathbf{A}$ algebras. Our main result, \Cref{uqgln uqglm duality}, obtains an FFT for $\Uqgln$ by computing the multiplicity-free decomposition of a certain quantized exterior algebra under the $\Uqgln \otimes \Uqglm$-action induced by embeddings of $\Uqgln$ and $\Uqglm$ into a quantized Clifford algebra. \Cref{quantum group embeddings into clqnm} summarizes the construction of \Cref{q skew duality type a}.

We begin by re-proving the classical skew $GL_n \times GL_m$-duality \Cref{skew gln glm duality}. This theorem is well-known and may be found in various references, e.g., see Theorem~4.1.1 in \cite{howe1995} or Theorem~38.2 in \cite{Bump}. We provide an alternative proof using a Clifford algebra. In particular,  \Cref{classical skew duality A} constructs commuting embeddings of $\gln$ and $\glm$ into the Clifford algebra $\Clnm$, as illustrated in the following diagram, and then computes a multiplicity-free decomposition of the exterior algebra $\bigwedge(\mathbb{C}^{nm})$ as a module of the tensor product of enveloping algebras $U(\gln) \otimes U(\glm)$. We describe this novel method because it generalizes to the quantum case, allowing us to prove our skew $\Uqgln \otimes \Uqglm$-duality \Cref{uqgln uqglm duality}. Howe does not provide these Clifford embeddings explicitly in \cite{howe1995}, but they are crucial to the generalization to Types $\mathbf{BD}$ developed in \crossrefsonch.
\begin{equation}\label[diag]{classical embeddings diagram}
	\begin{tikzcd}[column sep=-2em, row sep=tiny]
		& \bigwedge(\mathbb{C}^m)^{\otimes n} 
			\cong 
		  \bigwedge(\mathbb{C}^{nm}) 
		    \cong 
		  \bigwedge(\mathbb{C}^n)^{\otimes m}
		& \\[-0.5em]
		& \text{\Large$\circlearrowleft$}
		& \\ [-0.5em]
		& \Clnm
		& \\[+1.75em]
		& \glnm \ar[u]
		& \\[0.5em]
		\mathfrak{gl}_n \ar[ur] \ar[rr, no head, no tail, dashed] \ar[uur, "\lambda"]
		&
		& \mathfrak{gl}_m \ar[ul] \ar[uul, "\rho", swap]
	\end{tikzcd}
\end{equation}
Here $Cl(V \oplus V^*)$ denotes the Clifford algebra generated by the complex space $V \oplus V^*$ equipped with the canonical symmetric bilinear form arising from the duality pairing between $V$ and $V^*$\longer{, as in \Cref{symm bilinear form}}. When necessary, we view the Clifford algebra as a Lie algebra with bracket given by the usual algebra commutator.

With \Cref{skew gln glm duality} in hand, we turn to our main result: the skew $\Uqgln \otimes \Uqglm$-duality \Cref{uqgln uqglm duality}. The first step is to obtain commuting actions of $\Uqgln$ and $\Uqglm$ on the braided exterior algebra $\bigwedge_q(V^{(nm)})$, with $V^{(p)}$ denoting the natural $U_q(\mathfrak{gl}_p)$-module. Braided exterior algebras were introduced in \cite{berenstein} as $\Uqg$-module analogues of the exterior algebra $\bigwedge(V)$. Both the classical and the braided exterior algebras can be understood as quotients of a tensor algebra modulo an ideal generated by ``symmetric'' $2$-tensors. In the classical case, the ``symmetric'' $2$-tensors are always the $+1$ eigenvectors of the trivial flip maps that transpose tensor factors; in the quantum case, they are the eigenvectors with eigenvalue of the form $+q^r$ for some $r \in \mathbb{Q}$ of the braiding operators induced by the $R$-matrix of $\Uqg$. \Cref{braided_ext_alg} recalls the construction of $\bigwedge_q(\Vn)$ in detail.

\Cref{commuting quantum actions section} obtains quantum group actions on $\bigwedge_q(V^{(nm)})$ by emulating the classical Clifford algebra construction we develop in \Cref{classical skew duality A}. In particular, we define quantum analogues $\lambda_q$ and $\rho_q$ of $\lambda$ and $\rho$ as illustrated in the following diagram.
\begin{equation}\label[diag]{quantum group embeddings into clqnm}
	\begin{tikzcd}[row sep=small, column sep=-1.5em]
		& \bigwedge_q(V^{(nm)}) \cong \bigwedge_q(\Vn)^{\otimes m}
		& \\[-9pt]
		& \text{\Large$\circlearrowleft$}
		& \\ [-9pt]
		& Cl_q(nm)
		& \\[7pt]
		\Uqgln \ar[ur, "\lambda_q"] \ar[rr, no head, no tail, dashed]
		&
		& \Uqglm \ar[ul, "\rho_q", swap]
	\end{tikzcd}
\end{equation}
In the top row we have an isomorphism of $\Uqgln$-modules. We use $Cl_q(nm)$ to denote the quantum Clifford algebra\longer{ of \crossrefcliff{clqnk defn}, with twist $k = 1$}. The quantized Clifford algebra was first defined by Hayashi in \cite{hayashi_1990}. In this \chorpaper\ we use a similar version due to Kwon \cite{kwon_2014}. In \crossrefcliffch\ we recall basic properties of the classical Clifford algebras and study their quantized counterparts in depth: we obtain new results on their algebraic structure and representation theory, including a center calculation, a factorization as a tensor product, and a complete list of irreducible representations. \Cref{uqgln clqnm embedding} shows that $\lambda_q$ factors through $\Uqglnm$, much like $\lambda\colon \gln \to \Clnm$ factors through $\glnm$ in the classical case. 

In contrast to the classical picture of \Cref{classical embeddings diagram}, the map $\rho_q$ does not factor through $\Uqglnm$. In the quantum case, a single choice of weight basis cannot simultaneously describe both the isomorphism of $\Uqgln$-modules $\bigwedge_q(V^{(nm)}) \cong \bigwedge_q(\Vn)^{\otimes m}$ and the isomorphism of $\Uqglm$-modules $\bigwedge_q(V^{(nm)}) \cong \bigwedge_q(V^{(m)})^{\otimes n}$: the very construction of $\bigwedge_q(V^{(p)})$ depends on a choice of weight basis of $V^{(p)}$ through the $R$-matrix of $U_q(\mathfrak{gl}_p)$, which is defined in terms of its action on a given basis. \Cref{column major order means glm action does not factor} explains that, as illustrated in \Cref{root vector action on basis}, the action of simple root vectors in $\Uqglnm$ on a chosen $\Uqglnm$-weight basis of $\bigwedge_q(V^{(nm)})$ does not simultaneously ``align'' with the action simple root vectors in both $\Uqgln$ and $\Uqglm$.


Regardless, $\lambda_q$ and $\rho_q$ still define commuting actions on the $Cl_q(nm)$-module $\bigwedge_q(V^{(nm)})$. In fact, \Cref{uqgln uqglm embeddings commute} shows that $\Uqgln$ and $\Uqglm$ generate mutual commutants in $\End(\bigwedge_q(V^{(nm)}))$. Although $\Uqgln$ and $\Uqglm$ induce commuting subalgebras of \textit{module endomorphisms}, the images of $\lambda_q$ and $\rho_q$ do not define commuting subalgebras in $Cl_q(nm)$. This discrepancy is explained by the fact that the $Cl_q(nm)$-module $\bigwedge_q(V^{(nm)})$ is \textit{not} faithful\longer{, implied by} \crossrefcliff{clqnk is semisimple}. This is contrary to the classical case, where the representation $\bigwedge(\mathbb{C}^{nm})$ of the simple algebra $\Clnm$ is faithful and the images of $\lambda$ and $\rho$ indeed define commuting subalgebras in $\Clnm$.

We conclude our proof of \Cref{uqgln uqglm duality} by computing joint highest weight vectors in $\bigwedge_q(V^{(nm)})$ with respect to the $\Uqgln \otimes \Uqglm$-action induced by $\lambda_q \otimes \rho_q$ and practicing the \textit{double-commutant yoga} introduced by Howe in \cite{howe1995}.

We note that Theorem $6.16$ in \cite{lzz_2010} proves a similar related result. It computes the multiplicity-free decomposition of $\bigwedge_q(\Vn \otimes V^{(m)})$ as a $\Uqgln \otimes \Uqglm$-module. The braided exterior algebra $\bigwedge_q(\Vn \otimes V^{(m)})$ is constructed using the braiding induced by the tensor product of the $R$-matrices of $\Uqgln$ and $\Uqglm$. While $\bigwedge_q(\Vn \otimes V^{(m)})$ is isomorphic as a vector space to the braided exterior algebra $\bigwedge_q(V^{(nm)})$ considered here, which is constructed using the braiding of $\Uqglnm$ instead, they have different algebra structures. The structure considered here yields an isomorphism $\bigwedge_q(V^{(nm)}) \cong \bigwedge_q(\Vn)^{\otimes m}$ of $\Uqgln$-module algebras that motivates an extension to orthogonal types through the \textit{seesaw} developed in \crossrefsonch.

Although \cite{lzz_2010} also defines a multiplicity-free action of $\Uqgln \otimes \Uqglm$ on a quantized exterior algebra, it does not consider a double centralizer property inside a quantized Clifford algebra. In our work, Clifford algebras are essential. They provide a unifying framework for skew Howe duality results: in each case we consider certain subalgebras of a Clifford algebra defining commuting actions on a skew-symmetric space. In addition, the Clifford algebras framework elucidates a correspondence between the classical and quantized cases, in the sense that the operators induced by generators of quantized enveloping algebras de-quantize to their classical counterparts in the limit $q \to 1$.

The extension to types $\mathbf{BD}$ described in \crossrefsonch\ is a particularly remarkable feature of our method. In that setting, the Clifford algebras framework facilitates a computation of explicit joint highest weight vectors in $S^{\otimes m}$, with $S \cong \bigwedge_q(\Vn)$ denoting the spinor representation of $\Uqodn$. The description of $\End_{\Uqodn}(S^{\otimes m})$ obtained in \crossrefsonch\ yields matrix solutions of the Yang-Baxter equation\longer{ in \Cref{braiding ch}}. Specialized to certain roots of unity, these solutions define braid group representations arising from the fusion of so-called \textit{metaplectic} anyons, which have been studied previously by Rowell and Wang, and by Rowell and Wenzl, in \cite{rowell_wang_2011} and \cite{rowell_wenzl_2017}.

We note that there is a considerable literature studying quantized FFTs in the setting of \textit{symmetric} variables. Perhaps the first result is in \cite{zhang_howe_duality}. Most recently, \cite{lss_weyl_alg} obtains a double centralizer property inside a quantized Weyl algebra by considering actions of $\Uqgln$ and $\Uqglm$ on the quantized coordinate ring of the $n \times m$ matrices. 

In short, this \chorpaper\ is organized as follows. \Cref{classical skew duality A} deals with the classical case while \Cref{q skew duality type a} discusses the quantized generalization. Each section has three subsections: the first studies actions of the classical and quantum $\gln$ on an appropriate exterior algebra that factor through a Clifford algebra, the second considers commuting actions on a tensor product of exterior algebras, and finally the third achieves a multiplicity-free decomposition using the actions just defined.

\subsection*{Acknowledgements} 

This article emerged as part of the author's dissertation work under the supervision of Dan Bump. The author would like to thank Dan Bump for his infinite patience and support, and for the continuous stream of advice that made this work possible. This article benefits from joint work with Travis Scrimshaw on quantum Clifford algebras, which is in preparation, and from his relevant code available on SAGE \cite{sagemath}. 

\section{Notation and conventions}
\refstepcounter{notation}
\label{not and conv}

We will use $e_i$ throughout to denote the standard basis of $\mathbb{R}^n$.

In this section we fix our notation for Lie algebras and recall some definitions. \longer{Let $\lieg$ denote either the reductive complex Lie algebra $\mathfrak{gl}_n$, or a finite-dimensional complex semisimple Lie algebra. We will mainly consider the cases $\lieg = \mathfrak{gl}_n, \mathfrak{so}_n$, and $\mathfrak{sp}_n$.} 

\paragraph{Lie algebra presentation}
Following \cite{chari_pressley_1994}, we use the \textit{Chevalley presentation} throughout. If $\lieg$ is a complex semisimple Lie algebra, we let $A = [a_{ij}]$ denote its \textit{generalized Cartan matrix}.  We let $D = \mathrm{diag}(d_1, \ldots, d_r)$ denote the diagonal matrix of root lengths such that $DA$ is symmetric positive definite. Recall $a_{ii} = 2$ and $a_{ij} \leq 0$ whenever $i\neq j$.
	
The semisimple algebra $\lieg$ is generated by $H_i, E_i$, and $F_i$, for $i = 1, \ldots, r$, subject to the relations
\begin{equation}\label{chevalley presentation of g}
	\begin{gathered}
	[H_i, H_j] = 0, \quad [H_i, E_j] = a_{ij} E_j, \quad [H_i, F_j] = -a_{ij} F_j, \quad [E_i, F_j] = \delta_{ij} H_i,  \\
	\end{gathered}
\end{equation}
along with the \textit{Serre relations} for $i \neq j$:
\begin{equation}\label{Serre rels of g}
	(\mathrm{ad}_{E_i})^{1-a_{ij}}(E_j) = 0, \quad \text{and} \quad (\mathrm{ad}_{F_i})^{1-a_{ij}}(F_j) = 0
\end{equation}
\begin{remark}
	Authors like Jantzen \cite{J}, Lusztig \cite{lusztig_1988}, and Sawin \cite{sawin_1996} prefer the notation $X_i^+ = E_i$ and $X_i^- = F_i$, suggestive of positive and negative roots in general, and of ``raising'' and ``lowering'' operators in this context.
\end{remark}

We will denote by $\Ug$ the \textit{universal enveloping algebra} associated to $\lieg$\longer{ \cite[Chapter~10]{Bump}. This associative unital algebra has the same generators as $\lieg$, subject to relations in \eqref{chevalley presentation of g}, where in this case $[A, B] = AB - BA$ denotes the usual algebra commutator. The generators are also subject to the Serre relations \eqref{Serre rels of g}, which are written in this context as
\begin{equation*}\label{classical Serre relations}
	\sum_{k=0}^{1-a_{ij}}(-1)^k \binom{1-a_{ij}}{k} (X_i)^k X_j (X_i)^{1- a_{ij}-k} = 0.
\end{equation*}
In the last equality $X$ is either an $E$ or an $F$}.

Now fix the \textit{Cartan subalgebra} $\mathfrak{h}$ generated by the $H_i$\longer{ and let $\mathfrak{b}$ denote the \textit{Borel subalgebra} containing $\mathfrak{h}$}. The \textit{simple roots} of $\lieg$ are the linear functionals $\alpha_i\colon \mathfrak{h} \to \mathbb{C}$ satisfying
\begin{equation*}
	\alpha_i(H_j) = a_{ji}.
\end{equation*}
\longer{We use $\varPhi$ to denote the set of simple roots and we set 
\begin{equation*}
	\mathcal{Q} = \bigoplus_{\alpha \in \varPhi} \mathbb{Z} \alpha 
	\quad \text{and} \quad 
	\mathcal{Q}^+ = \bigoplus_{\alpha \in \varPhi} \mathbb{Z}_+ \alpha.
\end{equation*}}
We let
\begin{equation}\label{weights and dominant weights}
	\mathcal{P} = \{ \lambda \in \mathfrak{h}^* \mid \lambda(H_i) \in \mathbb{Z} \} \quad \text{and} \quad \mathcal{P}^+ = \{ \lambda \in \mathfrak{h}^* \mid \lambda(H_i) \in \mathbb{Z}_+ \}
\end{equation}
denote the lattice of \textit{weights} and \textit{dominant weights}, respectively.

\paragraph{Bilinear form}
When $\lieg$ is semisimple, there exists a unique non-degenerate symmetric invariant bilinear form $\langle , \rangle\colon \lieg \times \lieg \to \mathbb{C}$ such that
\begin{gather*}
\begin{split}
	\langle H_i, H_j \rangle = \inv{d_j} a_{ij}, \quad \langle H_i, E_j \rangle = \langle H_i, F_j \rangle = 0, \\
	\langle E_i, E_j \rangle = \langle F_i, F_j \rangle =  0,  \quad \text{and} \quad \langle E_i, F_j \rangle = \inv{d_i} \delta_{ij},
\end{split}
\end{gather*}
for all $i, j$ \cite[Theorem~2.2]{kac_1985}. When $\lieg = \gln$ we take $\langle , \rangle$ to be the non-degenerate trace bilinear form of the natural representation.
	
The bilinear form $\langle , \rangle$ induces an isomorphism of vector spaces $\nu\colon \mathfrak{h} \to \mathfrak{h}^*$ under which $H_i$ corresponds to the \textit{coroot} $\alpha_i^\vee = \inv{d_i} \alpha_i$. The induced form on $\mathfrak{h}^*$ satisfies
\begin{equation}\label{inner prod on roots}
	\langle \alpha_i, \alpha_j \rangle = d_i a_{ij}.
\end{equation}
Since $a_{ii} = 2$, we must have $d_i = \langle \alpha_i, \alpha_i \rangle / 2$. Thus we obtain a formula for the \textit{Cartan integers}:
\begin{equation}\label{cartan ints as inn prod}
	a_{ij} = \langle \alpha_i^\vee, \alpha_j \rangle 
	= { 
		2 \langle \alpha_i^\vee, \alpha_j \rangle 
		\over 
		\langle \alpha_i, \alpha_i \rangle
	  }.
\end{equation} 
Notice the bilinear form induced by $\nu$ is normalized so that 
$$\langle \alpha_i, \alpha_i \rangle = 2d_i.$$
When $\lieg$ is simple\longer{, the Cartan matrix $A$ has entries $a_{ij} \in \{-3, -2, -1, 0, 2\}$, so without loss of generality we may assume that $D$ has entries $d_i \in \{1, 2, 3\}$, which implies} $\langle \alpha_i, \alpha_i \rangle = 2$ for \textit{short} roots.

For concreteness and convenience, we record some relevant Cartan matrices: 
\begin{align}\label{cartan mats}
\begin{gathered}
	A_n = 
	\begin{bmatrix*}[r]
	2  & -1     &        &        & \\
	-1 & 2      & -1     &        & \\
	   & \ddots & \ddots & \ddots & \\
	   &        & -1     & 2      & -1 \\
	   &        &        & -1     & 2
	\end{bmatrix*}, 
	\\ \\
	B_n = 
	\begin{bmatrix*}[r]
	\begin{matrix}
	& & &&\\
	& A_{n-1} &&& \\
	& & &&
	\end{matrix} \rvline
	& 
	\begin{matrix*}[c]
	0 \\ \vdots \\ 0 \\ -1
	\end{matrix*} \\
	\cmidrule(lr){1-1}
	\begin{matrix*}[l]
	0 & \cdots & 0 & -2
	\end{matrix*}
	& 2
	\end{bmatrix*}, 
	\quad \text{and} \quad 
	D_n = 
	\begin{bmatrix*}[r]
	\begin{matrix}
	& & &&&\\
	& A_{n-1} &&&& \\
	& & &&&
	\end{matrix} \rvline
	& 
	\begin{matrix*}[c]
	0 \\ \vdots \\ 0 \\ -1 \\ 0
	\end{matrix*} \\
	\cmidrule(lr){1-1}
	\begin{matrix*}[l]
	0 & \cdots & 0 & -1 & 0
	\end{matrix*}
	& 2
	\end{bmatrix*}.
\end{gathered}
\end{align}
We label the matrices by the Lie type of the root system to which they are associated. The corresponding diagonal   root lengths matrices are described by $d = (1, \ldots, 1)$ for the root systems of types $A_n, D_n$, and by $d = (2, \ldots, 2, 1)$ for type $B_n$.

\longer{If $\alpha \in \mathcal{Q}$, define the \textit{root space} 
$$\lieg_\alpha = \{ x \in \lieg \mid [h, x] = \alpha(h) x \text{ for all } h \in \mathfrak{h}\}.$$
The elements of $\Delta = \{ \alpha \in \mathcal{Q} \mid \alpha \neq 0, \lieg_\alpha \neq 0 \}$ are the \textit{roots} of $\lieg$. The intersection $\Delta^+ = \Delta \cap \mathcal{Q}^+$ denotes the set of \textit{positive roots}. The \textit{negative roots} are given by $\Delta^- = - \Delta^+$ and $\Delta = \Delta^+ \coprod \Delta^-$ \cite[Chapter 1]{kac_1985}. Distinct root spaces are orthogonal with respect to the bilinear form $\langle , \rangle$:
$$\langle \lieg_\alpha, \lieg_\beta \rangle \text{ if } \alpha \neq \beta \text{ and }  \langle \lieg_\alpha, \mathfrak{h} \rangle \text{ if } \alpha \neq 0$$
\cite[Theorem 2.2(c)]{kac_1985}. The Lie algebra $\lieg$ has a triangular decomposition with respect to its roots. As vector spaces, 
\begin{equation*}\label{triangular decomp g}
	\lieg = \big(\bigoplus_{\alpha < 0} \lieg_\alpha \big) \oplus \mathfrak{h} \oplus \big(\bigoplus_{\alpha > 0} \lieg_\alpha \big).
\end{equation*}}

We denote the half sum of the positive roots in $\mathfrak{h}^*$ by $\rho$. The \textit{Weyl vector} $\rho$ is uniquely characterized by $\langle \alpha_i^\vee, \rho \rangle = 1$. \longer{Thus its image $\rho^* \in \mathfrak{h}$ under $\inv{\nu}$ satisfies $\alpha_i(\rho^*) = d_i$. Concretely, if $b_i = \sum_j (\inv{A})_{ji} d_j$, then
\begin{equation}\label{weyl vector}
	\rho = \sum_{i} b_i \alpha_i^\vee \quad \text{and} \quad \rho^* = \sum_{i} b_i H_i.
\end{equation}

The \textit{fundamental weights} $\omega_i \in \mathfrak{h}^*$  are defined by $\omega_i(H_j) = \delta_{ij}$, or equivalently, $\langle \omega_i, \alpha_j^\vee \rangle = \delta_{ij}$. The $\omega_i$ form a basis of $\mathcal{P}$ and clearly
$$\alpha_i = \sum_j a_{ji} \omega_j.$$ 
Since $C = DA$ is symmetric $\alpha_i = \sum_j c_{ij} (\inv{d_j} \omega_j)$ and therefore
$$\omega_i = \sum_j (D\inv{A})_{ij} \alpha_j^\vee$$
because $C$ is invertible. Then
$\langle \omega_i, \omega_j \rangle = d_i (\inv{A})_{ij}$ follows from the definition of $\omega_i$. 

Now consider any weights $\lambda = \sum_i \lambda_i \omega_i$ and $\mu = \sum_i \mu_i \omega_i$, and notice  $\lambda(H_i) = \lambda_i$ and similarly $\mu(H_i) = \mu_i$. Therefore
\begin{equation}\label{scalar product of weights}
	\langle \lambda, \mu \rangle = \sum_{ij} \lambda_i \mu_j \langle \omega_i, \omega_j \rangle = \sum_{ij} (D\inv{A})_{ij} \lambda(H_i) \mu(H_j).
\end{equation}}

\paragraph{Casimir element}

Let $X_\alpha$ denote any basis of $\lieg$, and let $X^\alpha$ denote the corresponding dual basis with respect to $\langle \cdot, \cdot \rangle$. The next formula uniquely characterizes the \textit{Casimir element}
\begin{equation}\label{casimir elt}
C = \sum_{\alpha} X_{\alpha}  X^{\alpha}  = \sum_{\alpha} X^{\alpha}  X_{\alpha} \in \Ug.
\end{equation}
The Casimir element is in fact canonical and central \cite[Theorem~10.2]{Bump}. In addition, $C$ acts as the scalar 
\begin{equation}\label{casimir eigenvalues}
	\chi_{\lambda}(C) = \langle \lambda, \lambda + 2\rho\rangle
\end{equation}
on any irreducible $\Ug$-module with highest weight $\lambda$ \cite[Corollary~2.6]{kac_1985}.

\paragraph{Quantum groups and their representations} 

Throughout this work, the term ``quantum group'' refers to an associative Hopf algebra $\Uqg$ as presented in \cite[Definition~9.1.1]{chari_pressley_1994}. \longer{For convenience, we reproduce the definition below. 

Let $\lieg$ denote a semisimple Lie algebra and let $A$ denote its generalized Cartan matrix with the diagonal matrix $D = \mathrm{diag}(d_1, \ldots, d_r)$ such that $DA$ is symmetric. Let $q$ be \textit{indeterminate} and write $q_i = q^{d_i}$.
	The \textit{quantum group} $\Uqg$ is the associative unital algebra over $\mathbb{C}(q)$ generated by the elements $E_i, F_i, K_i$, and $\inv{K_i}$, for $i = 1, \ldots, r$, subject to the relations:
\begin{equation}\label{uq rels}
\begin{gathered}
K_i K_j = K_j K_i, \quad K_i \inv{K_i} = \inv{K_i} K_i = 1, \\
K_i E_j \inv{K_i} = q_i^{a_{ij}} E_j, \quad K_i 
F_j \inv{K_i} = q_i^{-a_{ij}} F_j,  \\
E_i F_j - F_j E_i = \delta_{ij} \frac{K_i - \inv{K_i}}{q_i - \inv{q_i}},
\end{gathered}	
\end{equation}
together with the $q$-Serre relations for $i \neq j$,
\begin{equation}\label{uq Serre rels}
	\sum_{k = 0}^{1-a_{ij}} (-1)^k 
	\begin{bmatrix}
	1 - a_{ij} \\ k
	\end{bmatrix}_{q_i} 
	X_i^k X_j X_i^{1-a_{ij}-k} = 0.
\end{equation}
In the last equality every $X$ is either an $E$ or an $F$. The $q$-integer is given by $[n]_{q_i} = \frac{q_i^n - q_i^{-n}}{q_i - \inv{q_i}}$, the $q$-factorial $[n]_{q_i}! = [n]_{q_i}\cdots [1]_{q_i}$, and the $q$-binomial coefficient is defined analogously.} We note the following identity is often useful for verifying $q$-Serre relations when $a_{ij} = -1$. Writing $[A, B]_q = AB - qBA$, we have
\begin{align}\label{alt uq Serre rels}
\begin{split}
	\sum_{k = 0}^{1-a_{ij}} (-1)^k 
	\begin{bmatrix}
	1 - a_{ij} \\ k
	\end{bmatrix}_{q_i} 
	X_i^k X_j X_i^{1-a_{ij}-k} 
		&= X_i (X_i X_j - q_i X_j X_i) - \inv{q_i} (X_i X_j - q_i X_j X_i) X_i \\
		&= [X_i, [X_i, X_j]_{q_i}]_{\inv{q_i}}.
\end{split}
\end{align}

Unless otherwise stated, we assume that $\Uqg$ is equipped with the co-algebra structure defined by the comultiplication map $\Delta\colon \Uqg \to \Uqg^{\otimes 2}$ satisfying
\begin{align}\label{delta convention}
	\begin{split}
		\Delta(E_i) &= E_i \otimes K_i + 1 \otimes E_i \\
		\Delta(F_i) &= F_i \otimes 1 + \kinv \otimes F_i, 
		\quad \text{and} \\
		\Delta(K_i) &= K_i \otimes K_i,
	\end{split}
\end{align}
as in \cite{lzz_2010}. The choice of comultiplication convention is important in \ddichotomy{\Cref{type A duality ch,orthogonal duality ch}}{defining commuting module actions}, where different conventions are needed in order to ensure that relevant quantum groups induce commuting endomorphism subalgebras.

In this work, we only consider finite-dimensional type $(1, \ldots, 1)$ modules for the quantum groups $\Uqg$. Refer to Chapters~9 and~10 in \cite{chari_pressley_1994} for definitions. Every type $(1, \ldots, 1)$ finite-dimensional $\Uqg$-module is semisimple \cite[Theorem~5.17]{J} and each simple $\Uqg$-module is in one-to-one correspondence with a $\lieg$-module parametrized by a dominant weight of $\lieg$ \cite[Theorem~5.10]{J}. Moreover,  corresponding modules have the same weight multiplicities \cite[Lemma~5.14]{J}.

\section{Classical skew duality via Clifford algebras}
\label{classical skew duality A}

In this section we (re)-prove the classical skew $GL_n \times GL_m$-duality \Cref{skew gln glm duality} using a double centralizer property inside a Clifford algebra. This theorem is well-known and there are various proofs in the literature. We develop a Clifford algebra argument here because it generalizes to the quantum case, as illustrated by \Cref{classical embeddings diagram,quantum group embeddings into clqnm}, allowing us to prove our skew $\Uqgln \otimes \Uqglm$-duality \Cref{uqgln uqglm duality} in \Cref{q skew duality type a}. The enveloping algebra $U(\mathfrak{gl}_p)$ more closely resembles the quantum group $U_q(\mathfrak{gl}_p)$, so we work at the Lie algebra level throughout.

We prove \Cref{skew gln glm duality} in three steps. First, in \Cref{ext_alg_as_cln_module} we show that for any complex vector space $V$ the $\mathfrak{gl}(V)$-action by derivations on the exterior algebra $\bigwedge(V)$ factors through the spin action of the Clifford algebra $Cl(V \oplus V^*)$. Then we  construct commuting embeddings of $\gln$ and $\glm$ into $\Clnm$ in \Cref{commuting_a_embeddings}. In fact, first we embed $\gln$ and $\glm$ into $\glnm$ and then use the map $\glnm \to \Clnm$ of \Cref{ext_alg_as_cln_module}. Finally, we compute a multiplicity-free decomposition of $\bigwedge(\mathbb{C}^{nm})$ as a $\gln \otimes \glm$-module in \Cref{mf_decomposition_a}.

	\subsection{A $\gln$-action on the Clifford spin module}
	\label{ext_alg_as_cln_module}

In \crossrefcliff{cln props} we define the Clifford algebra $Cl(V \oplus V^*)$ and its spin action on the exterior algebra $\bigwedge(V)$ via inner and exterior multiplication operators $\iota_f$ and $\varepsilon_v$. For convenience, recall 
\begin{align}\label{inner ext mult ops gln review}
	\iota_f(\eta) 
	= \sum_{j=1}^k (-1)^{j-1} f(v_j) 
		w_1 \wedge \cdots \wedge \widehat{w_j} \wedge \cdots \wedge w_k
	\quad \text{and} \quad 
	\varepsilon_v(\eta) = v \wedge \eta
\end{align}
for any $v \in V$, $f \in V^*$, and $\eta = w_1 \wedge \cdots \wedge w_k \in \bigwedge(V)$. Our treatment is fairly standard and may be found in various sources. For instance, see \cite[Chapter~31]{Bump}, \cite[Chapter~6]{GW}, or \cite[Chapters~1-2]{michelsohn_lawson}. The inner multiplication operators are sometimes known as \textit{contractions} \cite[Chapter~6]{GW}. When necessary or convenient, we view the Clifford algebra as a Lie algebra with bracket given by the usual algebra commutator: $[A, B] = AB - BA$.

In this section we show that the natural $GL(V)$-action on $V$ induces a $\mathfrak{gl}(V)$-action on $\bigwedge(V)$ by derivations that factors through $Cl(V \oplus V^*)$. In particular, we construct a Lie algebra map $\Phi_n$ in \Cref{gln cl embedding}  making the following diagram commute.
\begin{equation}\label[diag]{gln factors through cl}
	\begin{tikzcd}
		Cl(V \oplus V^*)  \ar[r, "\cong"] & \End\left(\bigwedge(V)\right) \\
		\mathfrak{gl}(V) \ar[u, "\Phi_n"] \ar[ur] &
	\end{tikzcd}
\end{equation}

The constructions in this section generalize to the quantized setting. In particular, \Cref{braided_ext_alg} defines maps making \Cref{uqgln clq embedding compatibility diag} commute, which are key ingredients in the proof our quantized skew $\Uqgln \otimes \Uqglm$-duality \Cref{uqgln uqglm duality}.

It all starts with the natural $GL(V)$-action. The general linear group $GL(V)$ acts on $V$ by matrix multiplication and the diagonal maps $\delta^{(j)}(g) = g \otimes \cdots \otimes g$ extend this action to the tensor algebra $T(V) = \bigoplus_{j=0}^\infty V^{\otimes j}$. Differentiating the $GL(V)$-action on $T(V)$ by automorphisms yields a $\mathfrak{gl}(V)$-action by derivations. In fact, it yields an action of the universal enveloping algebra $U(\mathfrak{gl}(V))$ making the following diagram commute for each $j = 0, 1, 2, \ldots$. In the diagram, $\Delta\colon U(\mathfrak{gl}(V)) \to U(\mathfrak{gl}(V))^{\otimes 2}$ denotes the comultiplication of $U(\mathfrak{gl}(V))$.
\begin{equation}\label[diag]{GLV exponentiation diagram}
	\begin{tikzcd}[column sep=large]
		GL(V) \ar[r, "\delta^{(j-1)}"]
		& \End(V^{\otimes j}) \\
		\mathfrak{gl}(V) \ar[u, "\exp"] \ar[ur, "\Delta^{(j-1)}", swap]
		&
	\end{tikzcd}
\end{equation}

We define the exterior algebra $\bigwedge(V)$ and the induced $\mathfrak{gl}(V)$-action it carries by considering $U(\mathfrak{gl}(V))$ as a cocommutative Hopf algebra. We take the Hopf algebra point of view here in order to motivate the construction of the \textit{braided} exterior algebra in \Cref{braided_ext_alg}: in both cases we consider a tensor algebra modulo an ideal generated by certain eigenvectors of braiding operators induced by an $R$-matrix. The enveloping algebra $U(\mathfrak{gl}(V))$ is cocommutative, so it is (trivially) a \textit{braided bialgebra} \cite[Definition~VIII.2.2]{Kassel}. That is, $1 \otimes 1$ is an $R$-matrix of $U(\mathfrak{gl}(V))$. The braiding operator on $V \otimes V$ induced by $R = 1 \otimes 1$ is the flip map $\tau$ taking $v \otimes w \to w \otimes v$. Braiding operators are module maps, so the $\mathfrak{gl}(V)$-action on $T(V)$ preserves the set
\begin{equation}\label{classical sym2 def}
	 \mathrm{Sym}^2(V) = \{ v\otimes w + w \otimes v \mid v, w \in V\}
\end{equation}
of $+1$ eigenvectors of $\tau$. Therefore, the $\mathfrak{gl}(V)$-action descends to the \textit{exterior algebra}
\begin{equation}\label{exterior algebra def}
	\textstyle \bigwedge(V) \coloneqq T(V) / \langle \mathrm{Sym}^2(V) \rangle.
\end{equation}
Note that $\bigwedge(V)$ inherits the $\mathbb{Z}$-grading on $T(V)$ because $\langle \mathrm{Sym}^2(V) \rangle$ is homogeneous.

\Cref{braided_ext_alg} defines the braided exterior algebra $\bigwedge_q(\Vn)$, a quantum analogue of $\bigwedge(V)$, in a similar fashion, using braiding operators induced by the (non-trivial) $R$-matrix of the quantum group $\Uqgln$; compare \Cref{exterior algebra def} with \Cref{braided ext alg as tensor alg quotient}.

Identifying elements of $\mathfrak{gl}(V)$ with their counterparts in $V \otimes V^*$ allows us to factor the $\mathfrak{gl}(V)$-action on $\bigwedge(V)$ through $Cl(V \oplus V^*)$. Recall that there is a canonical isomorphism of vector spaces $\mathfrak{gl}(V) \cong \End(V) \cong V \otimes V^*$ satisfying
$$(v \otimes f)(w) = f(w) v, \quad v, w \in V, \,\, f \in V^*.$$
At the level of elementary linear algebra, this isomorphism says that square matrices are linear combinations of rank 1 outer products. It becomes an isomorphism of Lie algebras when we define a multiplication on $V \otimes V^*$ using the composition inside $\End(V)$ and take the Lie bracket defined by the usual algebra commutator. Concretely, the product $(v \otimes f) (w \otimes h)$ in $V \otimes V^*$ satisfies
$$\big((v \otimes f) \circ (w \otimes h)\big) (u) = h(u)f(w)v = \big(f(w) v \otimes h\big)(u), \quad u, v, w \in V, \,\, f, h \in V^*.$$
\longer{Lemma~II.2.5 in \cite{Kassel} explains this multiplication in terms of the \textit{evaluation map} $\mathrm{ev}_V\colon V \otimes V^* \to \mathbb{C}$ defined by $\mathrm{ev}_V(u \otimes f) = f(u)$. As discussed in \Cref{centralizers ribbon cat}, $\mathrm{ev}_V$ features in the definition of the category finite-dimensional $U(\mathfrak{gl}(V))$-modules as a \textit{ribbon category}.}\par

Since $\mathfrak{gl}(V)$ acts on $\bigwedge(V)$ by derivations, the action of any $X$ in $\mathfrak{gl}(V)$, identified with $u \otimes f$ in $V \otimes V^*$, factors into an inner and an exterior multiplication:
\begin{align}\label{glv action by derivations}
	\begin{split}
		X \rhd v_1 \wedge \cdots \wedge v_k 
		&= (Xv_1) \wedge v_2 \wedge \cdots \wedge v_k 
			+ \cdots +  
			v_1 \wedge \cdots \wedge (Xv_k)  \\
		&= u \wedge \left(\sum_{j=1}^k (-1)^{k-1} f(v_j) v_1 \wedge \cdots \wedge \widehat{v}_j \wedge \cdots v_k \right).
	\end{split}
\end{align}
Recall the action of inner and exterior multiplication operators $\iota_f$ and $\varepsilon_v$ as in \Cref{inner ext mult ops gln review}. In \Cref{glv action by derivations}, $\widehat{v}_j$ means that $v_j$ is omitted from the wedge product.

\Cref{glv action by derivations} is key. It shows that the element $X = v \otimes f$ in $\mathfrak{gl}(V) \cong V \otimes V^*$ acts on $\bigwedge(V)$ as the Clifford product
\begin{equation}\label{glv action by inner and exterior mult}
	X = \varepsilon_v\iota_f.
\end{equation}
\Cref{glv action by inner and exterior mult} defines a Lie algebra map $\mathfrak{gl}(V) \to Cl(V \oplus V^*)$ making \Cref{gln factors through cl} commute, thereby achieving this subsection's goal. The map is injective, so $\{\varepsilon_v\iota_f \mid v \in V, \; f \in V^*\}$ generates a subalgebra in $Cl(V \oplus V^*) \cong \End(\bigwedge(V))$ identical to $\mathfrak{gl}(V)$.

However, \Cref{glv action by inner and exterior mult} defines the map $\mathfrak{gl}(V) \to Cl(V \oplus V^*)$ canonically. We are ultimately interested in a quantum version of this embedding, and operators in the quantum setting are described by their action on a weight basis. Thus, in anticipation of the quantum case, we now describe the homomorphism defined by \Cref{glv action by inner and exterior mult} explicitly in terms of a $\mathfrak{gl}(V)$-weight basis. 

To this end, let $v_i$, for $i = 1, \ldots, n$, denote a $\mathfrak{gl}(V)$-weight basis of $V$ and let $v_j^*$ denote the corresponding dual basis of $V^*$, defined by $v_j^*(v_i)  = \delta_{ij}$. Then the
\begin{align}\label{bar v ext alg basis review}
	\bar{v}(\ell) = \begin{cases}
		v_1^{\ell_1} \wedge v_2^{\ell_2} \wedge \cdots \wedge v_n^{\ell_n}, & \text{if } \ell_i \in \{0, 1\} \\
		0, & \text{otherwise},
	\end{cases}
\end{align}
for $\ell \in \{0, 1\}^n$, form a basis of $\bigwedge(V)$\longer{ as in \Cref{cln props}}.

Our choice of weight basis defines an isomorphism $V \cong \mathbb{C}^n$; from now on, $\gln$ denotes $\mathfrak{gl}(V) \cong \mathfrak{gl}(\mathbb{C}^n)$. We work with the presentation of $\gln$ obtained from the Chevalley presentation of $\mathfrak{sl}_n$ described by \Cref{chevalley presentation of g,Serre rels of g} by adjoining elements $\bar{L}_i$, for $i = 1, \ldots, n$, subject to the relations
$$H_i = \bar{L}_i - \bar{L}_{i+1}, \quad \text{and} \quad [\bar{L}_i, \bar{L}_j] = 0.$$
The $\bar{L}_i$ correspond to a basis of $\mathfrak{h}^*$ that is orthonormal with respect to the trace bilinear form on $V$.
\begin{rmk}
	We use an overbar in this section to distinguish $\bar{v}(\ell)$ in $\bigwedge(V)$ from $v(\ell)$ in the braided exterior algebra $\bigwedge_q(\Vn)$ defined in \Cref{q skew duality type a}. Similarly, we use an overbar to distinguish the $\bar{L}_i$ in $\gln$ from the $\Uqgln$ generators introduced by \Cref{uqgln def}.
\end{rmk}

As in \Cref{glv action by inner and exterior mult}, we obtain the map $\gln \to \Cln$ by identifying elements of $\gln$ with their counterparts in $\mathbb{C}^n \otimes (\mathbb{C}^n)^*$. This time we identify generators \textit{only}. Let $M_{ij} v_k = \delta_{jk} v_i$ denote the matrix units with respect to the $v_i$ basis and recall that there is an isomorphism $\gln \to \mathrm{Mat}_n(\mathbb{C})$ satisfying
\begin{equation}\label{glv matrices}
	E_i \to M_{i,i+1}, \quad
	F_i \to M_{i+1, i}, 
	\quad \text{and} \quad
	\bar{L}_i \to M_{i,i}.
\end{equation}
Under this isomorphism, the $\bar{L}_i$ become matrices of unit trace. Composing with $\mathrm{Mat}_n(\mathbb{C}) \cong \mathbb{C}^n \otimes (\mathbb{C}^n)^*$ identifies the $\gln$ generators as
\begin{equation}\label{glv into v vstar}
	E_i \to v_i v_{i+1}^*, \quad
	F_i \to v_{i+1} v_i^*, 
	\quad \text{and} \quad
	\bar{L}_i \to v_i v_i^*.
\end{equation}

On the Clifford algebra side, we consider the generators $\psi_i = \iota_{v_i^*}$ and $\psi_i^\dagger = \varepsilon_{v_i}$, for $i = 1, \ldots, n$, as in \Cref{inner ext mult ops gln review}. The $\psi_i$ and $\psi_j^\dagger$ satisfy the canonical anticommutation relations
\begin{align}\label{cac review}
	\begin{gathered}
		\psi_i \psi_j + \psi_j \psi_i = \psi_i^\dagger \psi_j^\dagger + \psi_j^\dagger \psi_i^\dagger = 0, \\
		\psi_i^{\pdg} \psi_j^\dagger + \psi_j^\dagger \psi_i^{\pdg} = \delta_{ij},
	\end{gathered}
\end{align}
and they act on $\bigwedge(V)$ by lowering and raising operators: for any $\bar{v}(\ell)$,
\begin{align}\label{inner ext operator action review}
	\begin{split}
		\psi_i \, \bar{v}(\ell) &= (-1)^{\ell_1 + \cdots + \ell_{i-1}} \bar{v}(\ell - e_i), \text{ and} \\
		\psi_i^\dagger \, \bar{v}(\ell) &= (-1)^{\ell_1 + \cdots + \ell_{i-1}} \bar{v}(\ell + e_i).
	\end{split}
\end{align}

The next proposition describes the homomorphism $\gln \to \Cln$ explicitly by composing the map defined by \Cref{glv into v vstar} with the natural inclusion $\gamma\colon\mathbb{C}^n \otimes (\mathbb{C}^n)^* \hookrightarrow \Cln$ \longer{of \Cref{embedding base space into clifford alg}}.

\begin{prop}\label[prop]{gln cl embedding}
	Suppose $n = \dim V$. There is a homomorphism of Lie algebras $\Phi_n\colon \gln \to \Cln$ satisfying
	\begin{align*}
		E_i \to \psi_i^\dagger \psi_{i+1}^{\pdg}, \quad
		F_i \to \psi_{i+1}^\dagger \psi_i^{\pdg}, 
		\quad \text{and} \quad
		\bar{L}_i \to \psi_i^\dagger \psi_i^{\pdg},
	\end{align*}
	for each $E_i, F_i$, and $\bar{L}_i$ in $\gln$.
\end{prop}
\begin{proof}
	As explained above, this map is a composition of known Lie algebra maps: $\gln \cong \mathrm{Mat}_n(\mathbb{C}) \cong \mathbb{C}^n \otimes (\mathbb{C}^n)^* \subset \Cln$.
\end{proof}

This proposition is the classical precursor of \Cref{uqgln clqn embedding}.

	\subsection{Commuting $\gln$ and $\glm$ embeddings into the Clifford algebra}
	\label{commuting_a_embeddings}

Now suppose that $V = U \otimes W$ with $\dim U = n$ and $\dim W = m$. In this subsection, we construct commuting embeddings of $\mathfrak{gl}(U)$ and $\mathfrak{gl}(W)$ into the Clifford algebra
$$Cl\left((U \otimes W) \oplus (U \otimes W)^*\right) \cong \End\left(\bigwedge(U \otimes W)\right)$$
in \Cref{gln embedding into clnm,glm embedding into clnm}. These embeddings are depicted in \Cref{classical embeddings diagram} and they serve as the foundation for our proof the classical skew duality \Cref{skew gln glm duality}.

As in \Cref{ext_alg_as_cln_module}, the story begins with an action by matrix multiplication. The group $GL(U)$ acts on the tensor product $U \otimes W$ via its action on the first factor:
$$g \rhd (u \otimes w) = (gu) \otimes w, \qquad g \in GL(U), \,\, v \in U, \,\, w \in W.$$
Similarly, $GL(W)$ acts on $U \otimes W$ through its natural action on the second factor. These actions commute, so they induce a group homomorphism
$$GL(U) \times GL(W) \hookrightarrow GL(U \otimes W) = GL(V)$$
that defines commuting actions on the $GL(V)$-module $\bigwedge(V) = \bigwedge(U \otimes W)$. We compose the differentials of these actions with the map $\mathfrak{gl}(V) \to Cl(V \oplus V^*)$ defined by \Cref{glv action by inner and exterior mult} to obtain commuting actions of $\mathfrak{gl}(U)$ and $\mathfrak{gl}(W)$ on $\bigwedge(V)$ that factor through $Cl(V \oplus V^*)$.

Since we have an eye on the quantum case, we will define the embeddings of $\mathfrak{gl}(U)$ and $\mathfrak{gl}(W)$ into $Cl\left((U \otimes W) \oplus (U \otimes W)^*\right)$ explicitly with respect to a $\mathfrak{gl}(V)$-weight basis of $V = U \otimes W$. We start $U$ and $W$: let $u_i$ and $w_j$ denote $\mathfrak{gl}(U)$- and $\mathfrak{gl}(W)$-weight bases of $U$ and $W$, and let $u_i^*$ and $w_j^*$ denote the corresponding dual bases of $U^*$ and $W^*$. Then the products $u_i \otimes w_j$, for $i = 1, \ldots, n$ and $j = 1, \ldots, m$, define a $\mathfrak{gl}(V)$-weight basis of $V$ because
\begin{equation}\label{glv factors}
	\mathfrak{gl}(V) 
		\cong
	\End(U \otimes W)
		\cong 
	(U \otimes U^*) \otimes (W\otimes W^*) 
		\cong 
	\mathfrak{gl}(U) \otimes \mathfrak{gl}(W).
\end{equation}
Setting
\begin{equation}\label{vk basis}
	v_{i + (j-1)n} = u_i \otimes w_j
\end{equation}
arranges this basis in the \textit{column-major order} $u_1 \otimes w_1, \ldots u_n \otimes w_1$, $u_1 \otimes w_2, \ldots, u_n \otimes w_2$, $\ldots, u_1 \otimes w_m, \ldots, u_n \otimes w_m$ and defines isomorphism $\mathbb{C}^{nm} \cong U \otimes W$. We then obtain a $\mathfrak{gl}(V)$-weight basis of $\bigwedge(V) \cong \bigwedge(\mathbb{C}^{nm})$ by considering the $\bar{v}(\ell)$ as in \Cref{bar v ext alg basis review} for $\ell \in \{0, 1\}^{nm}$.

In the quantized setting of \Cref{commuting quantum actions section}, the isomorphism defined by \Cref{vk basis} allows us to consider the braided exterior algebra $\bigwedge_q(V^{(nm)})$ defined using the $R$-matrix of the single quantum group $\Uqglnm$, instead of the braided exterior algebra $\bigwedge_q(\Vn \otimes V^{(m)})$ used in \cite{lzz_2010}, which requires the $R$-matrices of both $\Uqgln$ and $\Uqglm$.

Having chosen explicit weight bases of $U$, $W$, and $V = U \otimes W$, we refer to $\mathfrak{gl}(U) \cong \mathfrak{gl}(\mathbb{C}^n)$, $\mathfrak{gl}(W) \cong \mathfrak{gl}(\mathbb{C}^m)$, and $\mathfrak{gl}(U \otimes W) \cong \mathfrak{gl}(\mathbb{C}^{nm})$ as $\gln$, $\glm$, and $\glnm$.

The Clifford algebra $\Clnm$ is generated by the inner and exterior multiplication operators $\psi_k = \iota_{v_k^*}$ and $\psi_k^\dagger = \varepsilon_{v_k}$ \longer{of \Cref{inner and ext mult maps}}, for $k = 1, \ldots, nm$, because the functionals $v_{i + (j-1)n}^* = u_i^* \otimes w_j^*$ define a basis of $(\mathbb{C}^{nm})^* \cong (U \otimes W)^*$ dual to $v_1, \ldots, v_{nm}$. As in \Cref{ext_alg_as_cln_module}, the $\psi_k$ and $\psi_k^\dagger$ satisfy the canonical anticommutation relations \eqref{cac review} and they act on $\bigwedge(U \otimes W)$ as in \Cref{inner ext operator action review}. In particular, the operator $\psi_{i + (j-1)n}$ \textit{vacates} the $(i,j)$ position in the $n \times m$ grid defined by $\bar{v}(\ell)$, while $\psi_{i + (j-1)n}^\dagger$ \textit{occupies} it.

\Cref{gln embedding into clnm,glm embedding into clnm} define the desired embeddings of $\gln$ and $\glm$ into $\Clnm \cong \End\left(\bigwedge(U \otimes W)\right)$. These maps rely crucially on the factorizations $\mathfrak{gl}(U \otimes W) \cong \mathfrak{gl}(U) \otimes \mathfrak{gl}(W)$ and $X = \varepsilon_u \iota_f$ of Relations \eqref{glv factors} and \eqref{glv action by inner and exterior mult}: the former takes an element in $\gln$ or $\glm$ into $\glnm$ by tensoring with the identity $I$ and then the latter identifies the resulting element in $\glnm \cong \mathbb{C}^{nm} \otimes (\mathbb{C}^{nm})^*$ with a Clifford product in $\Clnm$.

Therefore the first step is to describe the maps $\gln \to \glnm$ and $\glm \to \glnm$ explicitly. The isomorphism $\mathfrak{gl}(V) \cong \mathfrak{gl}(U) \otimes \mathfrak{gl}(W)$ of \Cref{glv factors} implies that $\mathfrak{gl}(U) \cong \mathfrak{gl}(U) \otimes \id_W \subset \mathfrak{gl}(V)$. Identifying the elements $X \otimes I$, with $X$ a generator in $\gln \cong \mathrm{Mat}_n(\mathbb{C}) \cong \mathbb{C}^n \otimes (\mathbb{C}^*)^n$, with their counterparts in $\mathrm{Mat}_{nm}(\mathbb{C}) \cong \glnm$ results in the map $\gln \to \glnm$. For instance, the simple positive root vector $E_i^{(n)}$ in $\gln$ is identified with $u_i \otimes u_{i+1}^*$ in $\mathbb{C}^n \otimes (\mathbb{C}^n)^*$, and the matrix $I$ is identified with the canonical vector $\sum_{j=1}^m w_j \otimes w_j^*$ in $\mathbb{C}^m \otimes (\mathbb{C}^m)^*$, so the map $\gln \to \glnm$ takes
\begin{align}\label{gln ei as matrix units}
	 \begin{split}
	 	E_i^{(n)} &\to \sum_{j=1}^m u_i \otimes u_{i+1}^* \otimes w_j \otimes w_j^* \\
	 	&\leftrightarrow \sum_{j=1}^m M_{i + (j-1)n, \, i+1 + (j-1)n} \\
	 	&\leftrightarrow \sum_{j=1}^m E_{i + (j-1)n}^{(nm)}.
	 \end{split}
\end{align}
As usual, the $M_{a b}$ denote matrix units defined by $M_{ab} v_{c} = \delta_{bc} v_{a}$. The third identity is a consequence of the isomorphism $\mathfrak{gl}_{nm}  \to \mathrm{Mat}_{nm}(\mathbb{C})$, as in \Cref{glv matrices}.

\begin{rmk}\label[rmk]{generator superscript}
	The superscript on a Chevalley generator indicates the algebra to which it belongs. For instance, $E_i^{(n)}$ is an element of $\mathfrak{gl}_n$ while $E_{i + (j-1)n}^{(nm)}$ lives in $\glnm$.
\end{rmk}

The matrix units in \Cref{gln ei as matrix units} shift certain occupied positions upwards, to a previous \textit{adjacent} position in the column-major order defined by \Cref{vk basis}. This action coincides with the effect of the simple root vectors in $\glnm$, as illustrated in \Cref{root vector action on basis}, so we obtain the following proposition. 

\begin{prop}\label[prop]{gln embedding into glnm}
	There is a Lie algebra homomorphism $\mathfrak{gl}_n \to \glnm$ mapping
	\begin{align*}
		X_i^{(n)} &\to \sum_{j=1}^m X_{i + (j-1)n}^{(nm)},
	\end{align*}
	where $X_a$ denotes one of $E_a$, $F_a$, or $\bar{L}_a$.
\end{prop}
\begin{proof}
As explained by \Cref{gln ei as matrix units}, this map is the composition 
	$\gln \cong \left(U \otimes U^*\right) \otimes \id_W \subset \End(U) \otimes \End(W) \cong \glnm$ of known Lie algebra maps.
\end{proof}

Dually, the factorization $\mathfrak{gl}(V) \cong \End(U) \otimes \End(W) \cong \mathfrak{gl}(U) \otimes \mathfrak{gl}(W)$ of \Cref{glv factors} implies that $\mathfrak{gl}(W) \cong \id_U \otimes \mathfrak{gl}(W) \subset \mathfrak{gl}(V)$. In this case $\id_U$ corresponds to $\sum_{i=1}^n u_i \otimes u_i^*$ in $\mathbb{C}^n \otimes (\mathbb{C}^n)^*$, so the map $\glm \to \glnm \cong \mathrm{Mat}_{nm}(\mathbb{C})$ satisfies
\begin{align}\label{glm ej as matrix units}
	\begin{split}
		E_j^{(m)} &\to \sum_{i=1}^n u_i \otimes u_i^* \otimes w_j \otimes w_{j+1}^* \\
		&\leftrightarrow \sum_{i=1}^n M_{i + (j-1)n, \, i + jn}.
	\end{split}
\end{align}

In this case, the matrix units in \Cref{glm ej as matrix units} shift certain occupied positions leftwards, to \textit{non}-adjacent positions in the column-major order defined by \Cref{vk basis}. These matrix units no longer correspond to \textit{simple} root vectors in $\glnm$; however, they are identified to non-simple root vectors given by nested commutators.

\begin{prop}\label[prop]{glm embedding into glnm}
	There is a Lie algebra homomorphism $\glm \to \glnm$ satisfying
	\begin{align*}
		E_j^{(m)} 
			&\to 
		\sum_{i=1}^n 
			\left[\big[[E_{i + (j-1)n}^{(nm)}, \, E_{i + 1 + (j-1)n}^{(nm)}], E_{i+2 + (j-1)n}^{(nm)}\big], \cdots, E_{i-1 + jn}^{(nm)} \right], 
		\quad \text{and} \\
		F_j^{(m)} 
			&\to 
		\sum_{i=1}^n 
			\left[\big[[F_{i-1+jn}^{(nm)}, \, F_{i-2 + jn}^{(nm)}], F_{i-3 + jn}^{(nm)}\big], \cdots, F_{i+(j-1)n}^{(nm)} \right].
	\end{align*}
\end{prop}
\begin{proof}
\Cref{glm ej as matrix units} explains that this map is the composition $\glm \cong \id_U \otimes \left(W \otimes W^*\right) \subset \End(U) \otimes \End(W) \cong \glnm$ of known Lie algebra maps.
\end{proof}

The next two propositions define the embeddings $\lambda$ and $\rho$ of $\gln$ and $\glm$ into $\Clnm$ by composing the maps defined in \Cref{gln embedding into glnm,glm embedding into glnm} with the map $\Phi_{nm}\colon \gln \to \Clnm$ of \Cref{gln cl embedding}. These maps make the following diagram commute and they motivate the quantized $\lambda_q \colon \Uqgln \to Cl_q(nm)$ and $\rho_q \colon \Uqglm \to Cl_q(nm)$ defined in \Cref{uqgln clqnm embedding,uqglm clqnm embedding}.

\begin{equation}\label[diag]{gln glm embeddings factor through glnm cd}
	\begin{tikzcd}[row sep={2.25cm,between origins}, column sep=large]
		\gln \ar[r, "\Delta^{(m-1)}"] \dar \ar[drr, dashed, blue!60, "\lambda"]
		& \gln^{\otimes m} \ar[r, "\Phi_n^{\otimes m}"]
		& \Cln^{\otimes m} \ar[r, "\cong"] \ar[d, "\, \,\Gamma_U"]
		& \End(\bigwedge(\mathbb{C}^n)^{\otimes m}) \ar[d, "\cong"]\\
		\gln \otimes \glm \ar[rr, "\Phi_{nm}"]
 		&
 		& \Clnm \ar[r, "\cong"]
 		& \End\left(\bigwedge(\mathbb{C}^{nm})\right) \\
 		\glm \ar[r, swap, "\Delta^{(n-1)}"]  \uar  \ar[urr, dashed, swap, blue!60, "\rho"]
 		& \glm^{\otimes n} \ar[r, swap, "\Phi_m^{\otimes n}"]
 		& Cl\left(\mathbb{C}^m \oplus (\mathbb{C}^m)^*\right)^{\otimes n} \arrow[u, swap, "\,\,\Gamma_W"] \ar[r, swap, "\cong"]
 		& \End\left(\bigwedge(\mathbb{C}^m)^{\otimes n} \right) \ar[u, swap, "\cong"]
 	\end{tikzcd}
\end{equation}
Recall that \crossrefcliff{cl tensor m into clnm} defines $\Gamma_U$ and $\Gamma_W$. 

\begin{prop}\label[prop]{gln embedding into clnm}
	Recall the superscript notation explained in \Cref{generator superscript}. There is a Lie algebra map $\lambda\colon \gln \to \Clnm$ satisfying
\begin{align*}
	E_i^{(n)} &\to \sum_{j=1}^m \psi^\dagger_{i + (j-1)n} \psi_{i+1 + (j-1)n}^{\pdg}, 
	\\
	F_i^{(n)} &\to \sum_{j=1}^m \psi^\dagger_{i+1 + (j-1)n} \psi_{i + (j-1)n}^{\pdg}, 
	\quad \text{and} \\
	L_i^{(n)} &\to \sum_{j=1}^m \psi_{i + (j-1)n}^\dagger \psi_{i + (j-1)n}^{\pdg}
\end{align*}
for every $E_j^{(n)}, F_j^{(n)}$, and $L_j^{(n)}$ in $\gln$.
\end{prop}
\begin{proof}
	This map is the composition of the homomorphism $\gln \to \glnm$ defined in \Cref{gln embedding into glnm} with the map $\Phi_{nm}\colon \glnm \to \Clnm$ defined in \Cref{gln cl embedding}. 
\end{proof}

\begin{prop}\label[prop]{glm embedding into clnm}
	There is a Lie algebra map $\rho\colon \glm \to \Clnm$ satisfying
	\begin{align*}
		E_j^{(m)} &\to \sum_{i=1}^n \psi^{\dagger}_{i + (j-1)n} \psi_{i + jn}^{\pdg}, 
		\\
		F_j^{(m)} &\to \sum_{i=1}^n \psi^{\dagger}_{i + jn} \psi_{i + (j-1)n}^{\pdg},
		\quad \text{and} \\
		L_j^{(m)} &\to \sum_{i=1}^n \psi_{i + (j-1)n}^\dagger \psi_{i + (j-1)n}^{\pdg},
	\end{align*}
	for every $E_j^{(m)}, F_j^{(m)}$, and $L_j^{(m)}$ in $\glm$.
\end{prop}
\begin{proof}
	Much like $\lambda$ of \Cref{gln embedding into clnm}, this map is the composition of the homomorphism $\glm \to \glnm$ defined in \Cref{glm embedding into glnm} with $\Phi_{nm}$. 
\end{proof}

\begin{rmk}\label[rmk]{column major order means glm action does not factor}
	The embeddings defined in \Cref{gln embedding into clnm,glm embedding into clnm} \textit{both} factor through $\glnm$, but the column-major ordering defined by \Cref{vk basis} makes it so that only the $\gln$ root vectors ``align'' with those of $\glnm$, as illustrated in \Cref{root vector action on basis}. In contrast, simple root vectors in $\glm$ correspond to non-simple root vectors in $\glnm$, which are implemented by highly nested commutators. In the quantized setting, \Cref{uqgln clqnm embedding,uqglm clqnm embedding} describe quantum analogues of the $\gln$ and $\glm$ embeddings, but \textit{only} the $\Uqgln$ embedding factors through $\Uqglnm$. In the quantum case, a single choice of $\Uqglnm$-weight basis cannot simultaneously describe both the isomorphism $\bigwedge_q(V^{(nm)}) \cong \bigwedge_q(\Vn)^{\otimes m}$ of $\Uqgln$-modules and the isomorphism $\bigwedge_q(V^{(nm)}) \cong \bigwedge_q(V^{(m)})^{\otimes n}$ of $\Uqglm$-modules.
\end{rmk}

Now we take a moment to offer some explanatory remarks. To help interpret our formulas, we introduce the shorthand $v_{ij} = v_{i + (j-1)n}$ and arrange the $v_{ij}$ in a rectangular array as follows:
\begin{equation}\label[diag]{basis rect diagram}
	\begin{array}{ccccc}
	v_{11} & v_{12} & v_{13} & \cdots & v_{1m} \\
	v_{21} & v_{22} & v_{23} & \cdots & v_{2m} \\
	\vdots & \vdots & \vdots & \ddots & \vdots \\
	v_{n1} & v_{n2} & v_{n3} & \cdots & v_{nm}
	\end{array}
\end{equation}
Then if $\ell \in \{0, 1\}^{nm}$, the wedge product $\bar{v}(\ell)$ in $\bigwedge(\mathbb{C}^{nm}) \cong \bigwedge(U \otimes W)$ describes a state of \textit{occupied} and \textit{vacant} positions in the $n \times m$ array \eqref{basis rect diagram}.

\Cref{root vector action on basis} illustrates the action of simple positive root vectors in $\gln$, $\glm$, and $\glnm$ with respect to the $\glnm$-weight basis defined by \Cref{vk basis}. Under $\lambda$ and $\rho$, each simple root vector $\gln$ and $\glm$ becomes the sum of a product of a creation and an annihilation operator, so the $E_i^{(n)}$, $F_i^{(n)}$, $E_j^{(m)}$, and $F_j^{(m)}$ preserve spaces of homogeneous degree. More specifically, if we let $\alpha_{i,j}^v = e_{i,j} - e_{i+1, j}$ and $\alpha_{i,j}^h = e_{i,j} - e_{i, j+1}$, we see that \Cref{gln embedding into clnm,glm embedding into clnm} imply that
	\begin{align}
		\lambda(E_i^{(n)}) \, \bar{v}(\ell) 
			&= \sum_{j=1}^m (-1)^{p_j} \bar{v}(\ell + \alpha_{i,j}^v)
			\label{gln ei on basis}
		\quad \text{and} \\
		\rho(E_j^{(m)}) \, \bar{v}(\ell) 
			&= \sum_{i=1}^n (-1)^{p_i'} \, \bar{v}(\ell + \alpha_{i,j}^h)
			\label{glm ej on basis}
	\end{align}
	for some integers $p_j, p_i'$. Notice that $\lambda(E_i^{(n)})$ attempts to shift each occupied box in the $(i+1)$st row \textit{upwards}, as illustrated by \Cref{action of ein}. Dually, $\rho(E_j^{(m)})$ attempts to shift each occupied position in the $(j+1)$st column \textit{leftwards}, as in \Cref{action of ejm}. Attempting to fill an occupied position, or vacate an empty one, annihilates the state.

\begin{figure}[!h]
\begin{center}
	\begin{tikzcd}[column sep=large]
		v_1
		& v_{1 + n} \ar[
						dddl,
						rounded corners,
						darkgreen,
						to path={ -- ([xshift=-2ex]\tikztostart.west)
							|- ([xshift=2ex]\tikztotarget.east)}
						]
		& v_{1 + 2n} \ar[l, bend right=45, swap, "E_{2}^{(m)}", crossing over, red] \ar[
						dddl,
						rounded corners,
						darkgreen,
						to path={ -- ([xshift=-2ex]\tikztostart.west)
							|- ([xshift=2ex]\tikztotarget.east)}
						]
		& \cdots \ar[
						dddl,
						rounded corners,
						darkgreen,
						to path={ -- ([xshift=-2ex]\tikztostart.west)
							|- ([xshift=2ex]\tikztotarget.east)}
						]
		& v_{1 + (m-1)n} \ar[
						dddl,
						rounded corners,
						darkgreen,
						to path={ -- ([xshift=-2ex]\tikztostart.west)
							|- ([xshift=2ex]\tikztotarget.east)}
						] \\
		v_2 \ar[u, bend right=45, darkgreen]
		& v_{2 + n} \ar[u, bend right=45, darkgreen]
		& v_{2 + 2n} \ar[l, bend right=45, swap, "E_{2}^{(m)}", crossing over, red] \ar[u, bend right=45, darkgreen]
		& \cdots \ar[u, bend right=45, xshift=10pt, darkgreen]
		& v_{2 + (m-1)n} \ar[u, bend right=45, darkgreen]\\
		\vdots \ar[u, bend left=45, "E_2^{(n)}", blue, crossing over,] \ar[u, bend right=45, darkgreen]
		& \vdots \ar[u, bend left=45, "E_2^{(n)}", blue, crossing over] \ar[u, bend right=45, darkgreen]
		& \vdots \ar[u, bend left=45, "E_2^{(n)}", blue, crossing over] \ar[u, bend right=45, darkgreen]
		& \ddots \ar[u, bend left=45, "E_2^{(n)}", blue] \ar[u, bend right=45, xshift=10pt, darkgreen]
		& \vdots \ar[u, bend left=45, "E_2^{(n)}", blue, {name=U, left, draw=red}] \ar[u, bend right=45, darkgreen]\\
		v_{n} \ar[u, bend right=45, darkgreen]
		& v_{2n} \ar[u, bend right=45, darkgreen]
		& v_{3n} \ar[l, bend right=45, swap, "E_{2}^{(m)}", crossing over, red] \ar[u, bend right=45, darkgreen]
		& \cdots \ar[u, bend right=45, xshift=10pt, darkgreen]
		& v_{nm} \ar[u, bend right=45, darkgreen]
	\end{tikzcd}
	\caption{Some $\gln$, $\glm$, and $\glnm$ root vectors acting on the $v_i$ weight basis of $V \cong U \otimes W$. The green arrows illustrate the action of $E_{i+(j-1)n}^{(nm)}$ generators.}
	\label{root vector action on basis}
\end{center}
\end{figure}
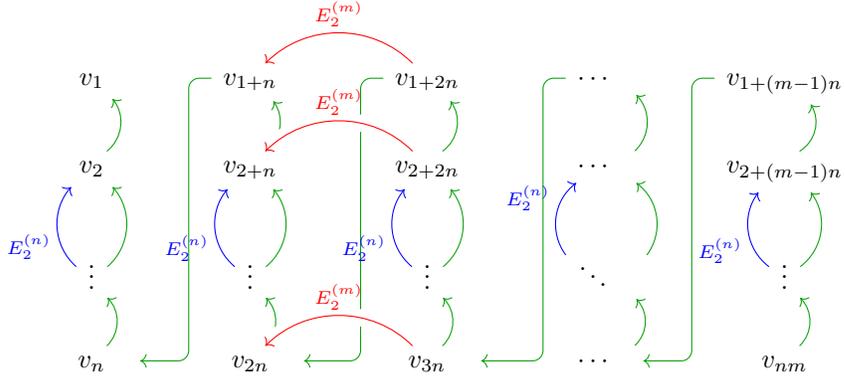
	
\begin{figure}[!h]
\begin{center}
\begin{tikzpicture}
\tikzset{every node/.style={scale=0.9}}
\node at (0, \yy) {\begin{varwidth}{5cm}{
	\begin{ytableau}
		*(green!20) v_{11} 
		& 
		& *(green!20) v_{13} 
		& *(green!20) v_{14} \\
		*(green!20) v_{21}
		& 
		& *(green!20) v_{23}
		&  \\
		*(green!20) v_{31}
		& *(green!20) v_{32}  
		& 
		& *(green!20) v_{34}  \\
	\end{ytableau}}\end{varwidth}};
\draw[->,semithick] (1.5,\yy) -- (3.25, \yy) 
	node[midway, above] {
	$\lambda\left(E_2^{(n)}\right)$
};
\node at (3.75,\yy) {\Large$0$};
\node at (4.25, \yy) {\Large$+$};
\node at (5.125, \yy) {\large$(-1)^3$};
\node at (7, \yy) 
{\begin{varwidth}{5cm}{
	\begin{ytableau}
		*(green!20) v_{11} 
		& 
		& *(green!20) v_{13} 
		& *(green!20) v_{14} \\
		*(green!20) v_{21} 
		&  *(blue!60) v_{22} 
		& *(green!20) v_{23}
		&  \\
		*(green!20) v_{31} 
		& *(blue!30) 
		& 
		&  *(green!20) v_{34} \\
	\end{ytableau}}\end{varwidth}};
\node at (8.625,\yy) {\Large$+$};
\node at (9.18, \yy) {\Large$0$};
\node at (9.75,\yy) {\Large$+$};
\node at (10.625, \yy) {\large$(-1)^7$};
\node at (12.5, \yy) 
{\begin{varwidth}{5cm}{
	\begin{ytableau}
		*(green!20) v_{11} 
		& 
		& *(green!20) v_{13} 
		& *(green!20) v_{14} \\
		*(green!20) v_{21}
		& 
		& *(green!20) v_{23}
		&  *(blue!60) v_{24} \\
		*(green!20) v_{31}
		& *(green!20) v_{32}  
		& 
		& *(blue!30) \\
	\end{ytableau}}\end{varwidth}};
\end{tikzpicture}
\caption{The $\gln$ root vector $E_2^{(n)}$ attempts to shift each occupied position in the $2$nd row \textit{upwards}. The first summand vanishes because $v_{11}$ is already occupied while the third and fourth terms vanish because $v_{23}$ and $v_{24}$ are vacant: there is nothing to shift. The sign of the $j$th term counts the number of occupied positions preceding $(2, j)$ in the column-major order defined by \Cref{vk basis}. Here $n = 3$ and $m = 4$.}
\label{action of ein}
\end{center}
\end{figure}

\begin{figure}[!h]
\begin{center}
\begin{tikzpicture}
\tikzset{every node/.style={scale=0.95}}
\node at (0, \yy) {\begin{varwidth}{5cm}{
	\begin{ytableau}
		*(green!20) v_{11} 
		& 
		& *(green!20) v_{13} 
		& *(green!20) v_{14} \\
		*(green!20) v_{21}
		& 
		& *(green!20) v_{23}
		&  \\
		*(green!20) v_{31}
		& *(green!20) v_{32}  
		& 
		& *(green!20) v_{34}  \\
	\end{ytableau}}\end{varwidth}};
\draw[->,semithick] (1.5,\yy) -- (3.5, \yy) 
	node[midway, above] {
	$\rho\left(E_2^{(m)}\right)$
};
\node at (4.5, \yy) {\large$(-1)^4$};
\node at (6.5, \yy) 
{\begin{varwidth}{5cm}{
	\begin{ytableau}
		*(green!20) v_{11} 
		& *(blue!60) v_{12} 
		& *(blue!30)
		& *(green!20) v_{14} \\
		*(green!20) v_{21}
		& 
		& *(green!20) v_{23}
		&  \\
		*(green!20) v_{31}
		& *(green!20) v_{32}  
		& 
		& *(green!20) v_{34}  \\
	\end{ytableau}}\end{varwidth}};
\node at (8.25,\yy) {\Large$+$};
\node at (9.25, \yy) {\large$(-1)^5$};
\node at (11.25, \yy) {\begin{varwidth}{5cm}{
	\begin{ytableau}
		*(green!20) v_{11} 
		& 
		& *(green!20) v_{13} 
		& *(green!20) v_{14} \\
		*(green!20) v_{21}
		& *(blue!60) v_{22}
		& *(blue!30)
		&  \\
		*(green!20) v_{31}
		& *(green!20) v_{32}  
		& 
		& *(green!20) v_{34}  \\
	\end{ytableau}}\end{varwidth}};
\node at (12.875,\yy) {\Large$+$};
\node at (13.5, \yy) {\Large$0$};
\end{tikzpicture}
\caption{The $\glm$ root vector $E_2^{(m)}$ attempts to shift each occupied position in the $3$nd \textit{column} \textit{to the left}. The third summand vanishes because the $(3,3)$ position is vacant. The sign of the $i$th term counts the number of occupied positions preceding $(i, 3)$ in the column-major order defined by \Cref{vk basis}. As in \Cref{action of ein}, $n = 3$ and $m = 4$.}
\label{action of ejm}
\end{center}	
\end{figure}

Now observe that each $L_i^{(n)}$ acts on $\bigwedge(U \otimes W)$ as a diagonal operator measuring degree of homogenous elements in the $i$th \textit{row}: for any $\bar{v}(\ell)$,
\begin{align}\label{L_i as degree operator}
	L_i^{(n)} \rhd \bar{v}(\ell) = \sum_{j=1}^m \psi_{i + (j-1)n}^\dagger \psi_{i + (j-1)n}^{\pdg} \bar{v}(\ell) = \sum_{j=1}^m \ell_{i + (j-1)n} \bar{v}(\ell).
\end{align}
Similarly, $L_j^{(m)}$ acts as the diagonal operator measuring degree of homogenous elements in the $j$th \textit{column}. \Cref{diag L action} illustrates the action $L_i^{(n)}$ and $L_j^{(m)}$ on the same given $\bar{v}(\ell)$. Since root vectors preserve components of homogeneous degree, their action on any state vector $v(\ell)$ must commute with that of every row and column degree operator $L_i^{(n)}$ and $L_j^{(m)}$. The representation $Cl\left((U \otimes W) \oplus (U \otimes W)^*\right) \cong \End\left(\bigwedge(V \otimes W)\right)$ is faithful, so commutation relations amongst module endomorphisms imply  commutation relations in the Clifford algebra.

\begin{figure}[!h]
\begin{center}
\begin{tikzpicture}
\tikzset{every node/.style={scale=0.95}}
\node at (1.1, \yy) {\begin{varwidth}{5cm}{
	\begin{ytableau}
		*(green!20) v_{11} 
		& 
		& *(green!20) v_{13} 
		& *(green!20) v_{14} \\
		*(green!20) v_{21}
		& 
		& *(green!20) v_{23}
		&  \\
		*(green!20) v_{31}
		& *(green!20) v_{32}  
		& 
		& *(green!20) v_{34}  \\
	\end{ytableau}}\end{varwidth}};
\draw[->,semithick] (3,\yy) -- (5, \yy) 
	node[midway, above] {
	$\lambda\left(\bar{L}_3^{(n)}\right)$
};
\node at (5.8, \yy) {\large$3$};
\node at (7.4, \yy) {\begin{varwidth}{5cm}{
	\begin{ytableau}
		*(green!20) v_{11} 
		& 
		& *(green!20) v_{13} 
		& *(green!20) v_{14} \\
		*(green!20) v_{21}
		& 
		& *(green!20) v_{23}
		&  \\
		*(blue!60) v_{31}
		& *(blue!60) v_{32}  
		& *(blue!30)
		& *(blue!60) v_{34}  \\
	\end{ytableau}}\end{varwidth}};

\def\yyy{-.8}
\node at (1.1, \yyy) {\begin{varwidth}{5cm}{
	\begin{ytableau}
		*(green!20) v_{11} 
		& 
		& *(green!20) v_{13} 
		& *(green!20) v_{14} \\
		*(green!20) v_{21}
		& 
		& *(green!20) v_{23}
		&  \\
		*(green!20) v_{31}
		& *(green!20) v_{32}  
		& 
		& *(green!20) v_{34}  \\
	\end{ytableau}}\end{varwidth}};
\draw[->,semithick] (3,\yyy) -- (5, \yyy) 
	node[midway, above] {
	$\rho\left(\bar{L}_4^{(m)}\right)$
};
\node at (5.8, \yyy) {\large$2$};
\node at (7.4, \yyy) {\begin{varwidth}{5cm}{
	\begin{ytableau}
		*(green!20) v_{11} 
		& 
		& *(green!20) v_{13} 
		& *(blue!60) v_{14} \\
		*(green!20) v_{21}
		& 
		& *(green!20) v_{23}
		& *(blue!30) \\
		*(green!20) v_{31}
		& *(green!20) v_{32}  
		& 
		& *(blue!60) v_{34}  \\
	\end{ytableau}}\end{varwidth}};
\end{tikzpicture}
\caption{The $\gln$ generator $\bar{L}_3^{(n)}$ measures the degree of a given state vector along the $3$rd \textit{row}. Dually, the $\glm$ generator $\bar{L}_4^{(m)}$ measures degree along the $4$th \textit{column}. As in \Cref{action of ein}, we take $n = 3$ and $m = 4$.}
\label{diag L action}
\end{center}	
\end{figure}
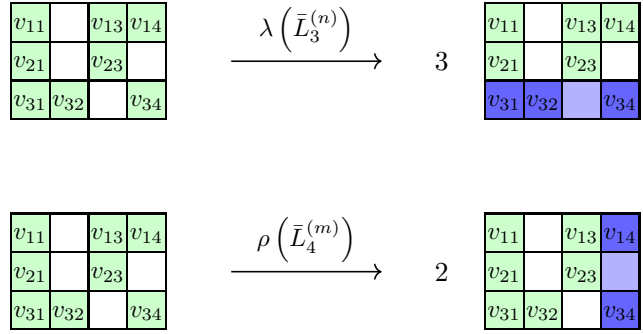

We conclude this subsection by noting that $\lambda$ and $\rho$ indeed define commuting subalgebras of $Cl\left((U \otimes W) \oplus (U \otimes W)^*\right) \cong \End\left(\bigwedge(\mathbb{C}^{nm})\right)$. This is guaranteed by the form of our embeddings, since they were obtained by tensoring each linear factor with an identity operator. In fact, even more is true: the images of $\lambda$ and $\rho$ generate each others full commutants.

\begin{prop}\label[prop]{gln and glm embeddings commute}
	The maps $\lambda$ and $\rho$ given in \Cref{gln embedding into clnm,glm embedding into clnm} define commuting subalgebras of $Cl\left((U \otimes W) \oplus (U \otimes W)^*\right)$. Moreover, $\mathfrak{gl}(U)$ and $\mathfrak{gl}(W)$ generate mutual commutants in $\End\left(\bigwedge(U \otimes W)\right)$.
\end{prop}
\begin{proof}
The first statement holds because the image of $\gln$ under $\lambda$ factors through $\gln \otimes \{I\} \subset \glnm$ while the image of $\glm$ under $\rho$ factors through $\{I\} \otimes \glm \subset \glnm$. 

Howe proves the second statement in \cite[Section~4.2.5]{howe1995}.
\end{proof}

	\subsection{Multiplicity-free decomposition of $\bigwedge(\mathbb{C}^{nm})$}
	\label{mf_decomposition_a}

We now recall the classical duality result for Type $\mathbf{A}$. Recall that every irreducible finite dimensional representation of a semisimple Lie group $G$ contains a \textit{highest weight vector} that is essentially unique (all highest weight vectors are scalar multiples of each other), and the highest weight vector identifies the module \cite[Theorem~22.3]{Bump}. The highest weight vector in an irreducible $G$-module is fixed by the unipotent subgroup of any Borel subgroup of $G$ \cite[Theorem~26.5]{Bump}. At the Lie algebra level, this means that the highest weight vector of an irreducible module is the unique weight vector that is annihilated by all the simple positive root vectors. In particular, a vector $v$ in some $\lieg$-module is a highest weight vector if and only if $v$ is a joint eigenvector of the $H_i$ and $E_i v = 0$ for each $i = 1, \ldots, r$. There is a one-to-one correspondence between highest weight vectors and \textit{dominant weights}, as specified by \Cref{weights and dominant weights}. When $\lieg = \gln$, the dominant weights $\mu \in \mathbb{Z}^n$ satisfy 
$$\mu_1 \geq \mu_2 \geq \cdots \geq \mu_n.$$ 
We refer the reader to Proposition~3.1.20 in \cite{GW} for more details. In what follows, we let $\pi_\mu^G$ denote the irreducible $G$-module with highest weight $\mu$.

\begin{thm}[Skew $GL_n \times GL_m$-duality]\label{skew gln glm duality}
	Consider the natural action of $GL(U) \times GL(W)$ on $V = U \otimes W$ and its extension to $\bigwedge(U \otimes W)$, as described by the commutative diagram \eqref{GLV exponentiation diagram}. As a $GL(U) \times GL(W)$-module, the exterior algebra $\bigwedge(U \otimes W)$ decomposes as
	$$\bigwedge(U \otimes W) \cong \sum_{\mu} \pi_{\mu}^{GL(U)} \otimes \pi_{\mu'}^{GL(W)}.$$
	Here, the sum ranges over all partitions $\mu$ with at most $n = \dim U$ rows of length at most $m = \dim W$: all partitions fitting inside an $n \times m$ rectangle.
\end{thm}

\begin{rmk}
	At the character level, skew $GL_n \times GL_m$-duality is the well-known \textit{dual Cauchy identity}, which states that if $\alpha_1, \ldots, \alpha_n$ and $\beta_1, \ldots, \beta_m$ are complex numbers of absolute value $< 1$ then
	$$
		\prod_{i = 1}^n \prod_{j = 1}^m (1 + \alpha_i \beta_j)
		=
		\sum_\lambda s_\lambda(\alpha_1, \ldots, \alpha_n) s_{\lambda'}(\beta_1, \ldots, \beta_m).
	$$
	The sum ranges over all partitions $\lambda$ fitting in an $n \times m$ box and each $s_\lambda$ is a Schur polynomial. For details, see e.g. \cite[Theorem~38.2]{Bump}.
\end{rmk}

There is a short proof of \Cref{skew gln glm duality} based on Schur Duality, due to Howe \cite[Theorem~4.1.1]{howe1995}. We provide an alternative proof at the end of this subsection, whose structure serves as a blueprint for establishing quantized duality results in \Cref {type a q mf decomposition} and in \crossrefson{q_mf_decomposition_bd}. Our alternative proof, outlined below, is independent of Schur Duality. Instead, it relies on a double commutant theorem of Weyl, which we reproduce below for the reader's convenience. \longer{Surprisingly, it is difficult to find a proof of \Cref{double comm} in the literature, so we include our own in \Cref{double comm appendix}.}

\begin{thm}[Double commutant]\cite[Theorem~B.1.1]{gurevich_howe}\label{double comm}
\begin{enumerate}[(a)]
	\item Let $\mathcal{H}$ be a finite dimensional vector space. Let $\mathcal{A}, \mathcal{A'}$ be two subalgebras of $\End(\mathcal{H})$ such that
	\begin{enumerate}[(1)]
		\item The algebra $\mathcal{A}$ acts semi-simply on $\mathcal{H}$.
		\item Each of $\mathcal{A}$ and $\mathcal{A'}$ is the full commutant of the other in $\End(\mathcal{H})$.
	\end{enumerate}
	Then $\mathcal{A'}$ acts semi-simply on $\mathcal{H}$, and as a representation of $\mathcal{A} \otimes \mathcal{A'}$, we have
	\begin{align}\label{mf decomp}
		\mathcal{H} \cong \bigoplus_{i \in I} \mathcal{H}_i \otimes \mathcal{H}_i',
	\end{align}
	where $\mathcal{H}_i$ are all the irreducible representations of $\mathcal{A}$, and $\mathcal{H'}_i$ are all the irreducible representations of $\mathcal{A'}$. In particular, we have a bijection between irreducible representations of $\mathcal{A}$ and $\mathcal{A'}$, and moreover, every isotypic component for $\mathcal{A}$ is an irreducible representation of $\mathcal{A} \otimes \mathcal{A'}$.

	\item On the other hand, if $\mathcal{A}$ and $\mathcal{A'}$ commute and the decomposition \eqref{mf decomp} holds as a representation of $\mathcal{A} \otimes \mathcal{A'}$, then each of $\mathcal{A}$ and $\mathcal{A'}$ is the full commutant of the other in $\End(\mathcal{H})$.
\end{enumerate}
\end{thm}

To establish \Cref{skew gln glm duality}, first we use \Cref{gln and glm embeddings commute} to prove that $GL(U)$ and $GL(W)$ generate mutual commutants in $\End\left(\bigwedge(U \otimes W)\right)$, thereby guaranteeing that the decomposition of the $GL(U) \otimes GL(W)$-module $\bigwedge(U \otimes W)$ is multiplicity-free by the so-called \textit{double commutant yoga} of \Cref{double comm}(a) \cite{howe1995}. Next, we use \Cref{hw bounds skew gln glm} to characterize the set of joint highest weights that can possibly appear in $\bigwedge(U \otimes W)$. Finally, we explicitly construct joint highest weight vectors in \Cref{hwv skew gln glm}, thereby exhibiting a \textit{Howe correspondence}, that is, a bijection $\mu \to f(\mu)$ such that the irreducible $GL(U) \times GL(W)$-module $\pi_U^\mu \otimes \pi_W^{f(\mu)}$ appears in the decomposition $\bigwedge(U \otimes W)$.

We are interested in developing a quantum analogue of \Cref{skew gln glm duality}, so we work at the Lie algebra level throughout. Showing that $GL(U)$ and $GL(W)$ generate full mutual commutants is equivalent to showing that $\mathfrak{gl}(U)$ and $\mathfrak{gl}(W)$ generate full mutual commutants: Proposition $3.14$ in \cite{LZ06} proves that if $(\pi, Z)$ is any representation of the connected complex Lie group $G$ and $(\mathrm{d}\pi, Z)$ is the corresponding representation of the Lie algebra $\lieg = \mathrm{Lie}(G)$, then
\begin{equation*}
	\End_{G}(Z) = \End_{\lieg}(Z).
\end{equation*}

\Cref{gln and glm embeddings commute} shows that $\mathfrak{gl}(U)$ and $\mathfrak{gl}(W)$ generate mutual commutants in $Cl\left((U \otimes W) \oplus (U \otimes W)^*\right)$.
Combining it with the Double Commutant \Cref{double comm} guarantees that the decomposition of the $GL(U) \times GL(W)$-module $\bigwedge(U \otimes W)$ is multiplicity-free.

The next proposition characterizes the set of highest weights that can possibly appear in the decomposition of $\bigwedge(U \otimes W)$.
\begin{lem}\label[lem]{hw bounds skew gln glm}
	If the irreducible $GL(U) \times GL(W)$-module $\pi_\mu^{GL(U)} \otimes \pi_\mu^{GL(W)}$ appears in the decomposition of $\bigwedge(U \otimes W)$ into irreducibles, then $\mu_1 \leq m$ and $\nu_1 \leq n$, with $n = \dim V$ and $m = \dim W$.
\end{lem}
\begin{proof}
Every weight of $GL(U) \times GL(W)$ is a weight of the induced $\mathfrak{gl}(U) \otimes \mathfrak{gl}(W)$-action. We show that the upper bounds in fact hold for any weights of the $\mathfrak{gl}(U) \otimes \mathfrak{gl}(W)$-module $\bigwedge(U \otimes W)$.

For every $\ell \in \{0,1\}^{nm}$, the vector $\bar{v}(\ell)$ is a simultaneous eigenvector of each $\bar{L}_i^{(n)}$ and each $\bar{L}_j^{(m)}$. \Cref{L_i as degree operator} shows that the $\bar{L}_i^{(n)}$-eigenvalue of $\bar{v}(\ell)$ is the degree of $\bar{v}(\ell)$ in the $i$th row, which is at most $m$ by construction. Similarly, the $\bar{L}_j^{(m)}$-eigenvalue of $\bar{v}(\ell)$ is its degree in the $j$th column, which is at most $n$. The $\bar{L}_i^{(n)}$ and $\bar{L}_j^{(m)}$ generate a Cartan subalgebra of $\mathfrak{gl}(U) \otimes \mathfrak{gl}(W)$, and there are $|\{0,1\}^{nm}| = 2^{nm} = \dim\left(\bigwedge(U \otimes W)\right)$ eigenvectors $\bar{v}(\ell)$, so it follows that every $\mathfrak{gl}(U)$-weight $\mu$ in $\bigwedge(U \otimes W)$ satisfies $\mu_1 \leq m$. Dually, every $\mathfrak{gl}(W)$-weight $\nu$ satisfies $\nu_1 \leq n$.
\end{proof}

The next proposition provides a supply of joint $GL_n \times GL_m$-highest weight vectors in $\bigwedge(U \otimes W)$, one for each partition $\mu$ fitting in an $n \times m$ rectangle. The highest weight vectors have a neat diagrammatical representation: the highest weight vector corresponding to $\mu$ is the wedge product of the $v_{ij}$, arranged as in \Cref{basis rect diagram}, filling the boxes occupied by the Young diagram corresponding to $\mu$. For example, the highest weight vector corresponding to $\mu = (3, 2, 1)$ is
\begin{equation*}
	(v_{11} \wedge v_{12} \wedge v_{13}) \wedge (v_{21} \wedge v_{22}) \wedge (v_{31}) = (v_{11} \wedge v_{21} \wedge v_{31}) \wedge (v_{12} \wedge v_{22}) \wedge (v_{13}),
\end{equation*}
the wedge product of the basis vectors illustrated in \Cref{hwv young diagram}.

\begin{figure}[!h]
\begin{center}
\begin{tikzpicture}
\node {\large$v_\mu =$};
\node at (1.75, 0) {\begin{varwidth}{5cm}{
	\begin{ytableau}
		v_{11} & v_{12} & v_{13} \\
		v_{21} & v_{22} \\
		v_{31}
	\end{ytableau}
}\end{varwidth}};
\end{tikzpicture}
\end{center}
\caption{The joint $GL_n \times GL_m$ highest weight vector corresponding to the partition $\mu = (3, 2, 1)$ is the wedge product of the basis vectors filling the boxes of the Young diagram defined by $\mu$.}
\label{hwv young diagram}
\end{figure}
Up to sign, this procedure determines a unique vector in $\bigwedge(U \otimes W)$ corresponding to a partition $\mu$ fitting in an $n \times m$ box. We note that \Cref{uqgln uqglm hwv} constructs joint highest weight vectors of the general linear quantum groups $\Uqgln$ and $\Uqglm$ using essentially the same procedure.

\begin{lem}\label[lem]{hwv skew gln glm}
	For $1 \leq \ell \leq n$, let $r_i(\ell) = v_{i1} \wedge \cdots \wedge v_{i\ell}$. Given a partition $\mu = (\mu_1, \ldots, \mu_n)$, define $v_{\mu} = r_1(\mu_1) \wedge \cdots \wedge r_n(\mu_n)$. The vector $v_{\mu}$ is a joint $GL_n \times GL_m$-highest weight vector in $\bigwedge(U \otimes W)$. Its weight as an eigenvector of the direct product of diagonal maximal tori $(T\times T') \subset GL_n \times GL_m$ is given by $\psi_\mu(t, t') = t^{\mu}( t' )^{\mu'}$. Here $t \in T$ is given by
	$$t = \begin{bmatrix}
		t_1 & & \\
		& \ddots & \\
		& & t_n
	\end{bmatrix},$$
	$t^\mu = \prod_i t_i^{\mu_i}$, and $(t')^{\mu'}$ is defined similarly.
\end{lem}
\begin{proof}
	As usual, we work at the Lie algebra level. Given \Cref{hw bounds skew gln glm}, it suffices to show that each $v_\mu$ is annihilated by each $E_i^{(n)}$ and each $E_j^{(m)}$. But this is immediate from the shape of $v_\mu$.
	
	Recall \Cref{gln ei as matrix units,glm ej as matrix units} and \Cref{root vector action on basis,action of ein,action of ejm}, which explain that the operators induced by the simple positive root vectors $E_i^{(n)}$ and $E_j^{(m)}$ in $\gln$ and $\glm$ shift certain occupied positions in $\bar{v}(\ell)$ above and to the left. Since $\mu$ is a left-justified partition, every position above and to the left of an occupied position in $v_\mu$ is occupied already, so 
	$$\lambda(E_i^{(n)}) \, v_\mu = \rho(E_j^{(m)}) \, v_\mu = 0.$$
	Hence, each $v_\mu$ is a joint highest weight vector for the $\mathfrak{gl}(U) \otimes \mathfrak{gl}(W)$-action.
\end{proof}

\begin{proof}[Proof of \Cref{skew gln glm duality}.]
	Together, \Cref{gln and glm embeddings commute} and \Cref{hwv skew gln glm,hw bounds skew gln glm} imply \Cref{skew gln glm duality}. \Cref{gln and glm embeddings commute} proves that $GL(U)$ and $GL(W)$ generate mutual commutants in $\End\left(\bigwedge(U \otimes W)\right)$, so the decomposition of $\bigwedge(U \otimes W)$ into $GL(U) \times GL(W)$ irreducibles is multiplicity-free by the Double Commutant \Cref{double comm}(a). \Cref{hw bounds skew gln glm} characterizes the set of weights that can possibly appear in the decomposition of $\bigwedge(U \otimes W)$, and \Cref{hwv skew gln glm} explicitly constructs a joint $GL(U) \times GL(W)$ highest weight vector for each dominant weight allowed by \Cref{hw bounds skew gln glm}.
\end{proof}


\section{Quantum skew duality via quantum Clifford algebras}
\label{q skew duality type a}
	
In this section we prove our skew $\Uqgln \otimes \Uqglm$-duality \Cref{uqgln uqglm duality}. \Cref{uqgln def} introduces general linear quantum groups by adjoining certain pairwise commuting generators to $\Uqsln$. Our proof of \Cref{uqgln uqglm duality} largely mirrors that of its classical counterpart, \Cref{skew gln glm duality}. Recall that to prove the classical skew duality \Cref{skew gln glm duality}, we first constructed the exterior algebra $\bigwedge(U \otimes W)$ as a module of the Clifford algebra $Cl\left((U\otimes W) \oplus (U\otimes W) ^*\right)$ in \Cref{ext_alg_as_cln_module}, and then we found commuting embeddings of $\mathfrak{gl}(U)$ and $\mathfrak{gl}(W)$ into $Cl\left((U\otimes W) \oplus (U\otimes W) ^*\right)$ in \Cref{commuting_a_embeddings}. Finally, we proved that $\mathfrak{gl}(U)$ and $\mathfrak{gl}(W)$ generate mutual commutants in $Cl\left((U\otimes W) \oplus (U\otimes W) ^*\right)$ and used the Double Commutant \Cref{double comm}(a) and a highest weight vector computation to conclude.

In the quantum case, we prove \Cref{uqgln uqglm duality} by first introducing the braided exterior algebra $\bigwedge_q(\Vn)$, a quantum analogue of $\bigwedge(U \otimes W)$ as a $Cl_q(nm)$-module in \Cref{braided_ext_alg}. The quantum Clifford algebra $Cl_q(nm)$ is the quantum analogue of $\Clnm$ studied in \crossrefcliffch. As mentioned in the introduction, $Cl_q(nm)$ is shorthand for $Cl_q(nm, 1)$: in this \chorpaper\ we use the quantized Clifford algebra introduced in \cite{kwon_2014}. Then we embed $\Uqgln$ and $\Uqglm$ into $Cl_q(nm)$ in \Cref{commuting quantum actions section}. Our homomorphisms define commuting subalgebras of module endomorphisms, so we compute highest weight vectors in \Cref{type a q mf decomposition} and appeal to the Double Commutant \Cref{double comm}(b) to conclude.

	\subsection{The braided exterior algebra as a quantum Clifford algebra module}
	\label{braided_ext_alg}

In this subsection we define the \textit{braided exterior algebra} $\bigwedge_q(V^{(n)})$ as a module for $Cl_q(n)$, a quantum analogue of the Clifford algebra $Cl\left(\mathbb{C}^n \oplus (\mathbb{C}^n)^*\right)$ studied in \crossrefcliffch. See also \cite{kwon_2014} for a definition. Braided exterior algebras were first introduced in \cite{berenstein} as $\Uqg$-module analogues of the classical $\lieg$-module $\bigwedge(V)$. Much like in the classical case, we then define an action of $\Uqgln$ on $\bigwedge_q(\Vn)$ that factors through $Cl_q(n)$ and makes the following diagram commute:
\begin{equation}\label[diag]{uqgln clq embedding compatibility diag}
	\begin{tikzcd}
		Cl_q(n)  \ar[r] & \End(\bigwedge_q(\Vn)) \\
		\Uqgln \ar[u, "\Phi_{q, n}"] \ar[ur] &
	\end{tikzcd}
\end{equation}
\Cref{uqgln def} introduces the general linear quantum group $\Uqgln$. \Cref{uqgln clq embedding compatibility diag} is the quantum analogue of \Cref{gln factors through cl}. In this case, the top arrow $Cl_q(n) \to \End(\bigwedge_q(\Vn))$ is \textit{not} an isomorphism: in the quantum case, the $Cl_q(n)$-module $\bigwedge_q(\Vn)$ is not faithful \crossrefcliff{clqnk is semisimple}.

For now, let $\lieg$ be a simple complex Lie algebra and let $V$ be any $\Uqg$-module. Later we will specialize to the case where $V$ is the natural $\Uqgln$-module. The quantum group $\Uqg$ acts on the tensor algebra $\mathcal{T}(V) \coloneqq \bigoplus_{j=0}^\infty V^{\otimes j}$ through its comultiplication. Unlike the classical enveloping algebra $\Ug$, the Hopf algebra $\Uqg$ is \textit{not} cocommutative, so the $\Uqg$-action on $\mathcal{T}(V)$ does \textit{not} preserve the ideal  generated by the set of symmetric tensors $\mathrm{Sym}^2(V)$ defined by \Cref{classical sym2 def}. Recall \Cref{exterior algebra def}, which explains that $\bigwedge(V)  \coloneqq T(V) / \big\langle \mathrm{Sym}^2(V) \big\rangle$.

Regardless, the category of $\Uqg$-modules is \textit{braided}\ddichotomy{, as explained in \Cref{uqg_modcat}}{\cite[Corollary~10.1.20]{chari_pressley_1994}}. 
For any pair of $\Uqg$-modules $V, W$, there is a natural isomorphism $\check{R}^{V, W}\colon V \otimes W \to W \otimes V$ satisfying
$$\check{R}^{V,W} (v \otimes w) = \tau (\mathcal{R} \rhd v \otimes w),$$
with $\mathcal{R}$ denoting the universal $R$-matrix of $\Uqg$. We view $\mathcal{R}$ as an invertible element of a suitable completion of $\Uqg \otimes \Uqg$ with a well-defined action on any tensor product of finite-dimensional $\Uqg$-modules. For details, see \ddichotomy{\Cref{uqg_modcat}}{\cite[Section~10.1.D]{chari_pressley_1994}}.
In the related topological Hopf algebra $\Uh$, the $R$-matrix has a unit constant term, so we may view the braiding operator as a deformation of the flip map $\tau$ \ddichotomy{in the sense of \Cref{braiding as deformation of inf braiding}}{\cite[Section~XX.4]{Kassel}}. The deformed map $\check{R} \coloneqq \check{R}^{V, V}$ is no longer involutive, but it satisfies the \textit{braid relations}
\begin{align}\label{check r satisfies braid rels}
\check{R}_i \check{R}_{i+1} \check{R}_i = \check{R}_{i+1} \check{R}_{i} \check{R}_{i+1}.
\end{align}
When $T$ is an operator on a tensor product $V \otimes V$, we write $T_i$ to denote the element of $\End(V^{\otimes n})$ acting by $T$ on the $i$th and $(i+1)$st factors:
$$T_i = \id_V^{\otimes (i-1)} \otimes T \otimes \id_V^{\otimes (n-1-i)}.$$
\Cref{check r satisfies braid rels} is ultimately a consequence of the \textit{quasitriangularity} axioms defining $\Uqg$ \cite[Definition~5.1]{majid_2003}, which guarantee $\mathcal{R}$ satisfies the celebrated \textit{Yang-Baxter equation}
$\mathcal{R}_{12} \mathcal{R}_{13} \mathcal{R}_{23} = \mathcal{R}_{23} \mathcal{R}_{13} \mathcal{R}_{12}$ \cite[Lemma~5.2]{majid_2003}. 

We use the braiding operator $\check{R}$ to define a quantum analogue $\mathrm{Sym}^2_q(V)$ of $\mathrm{Sym}^2(V)$ that is preserved by the $\Uqg$-action. Since the braiding operator $\check{R}$ belongs to the centralizer $\End_{\Uqg}(V \otimes V)$, the space of $\mu$-eigenvectors of $\check{R}$ is a $\Uqg$-module for each eigenvalue $\mu$. It is a well-known fact, see e.g. Corollary~2.22 in \cite{leduc_ram_1997} or Proposition~XVII.3.2 in \cite{Kassel}, that if $V_\mu$ is the irreducible $\Uqg$-module with highest weight $\mu$ and the decomposition $V_\mu \otimes V_\mu$ is multiplicity-free, then the braiding operator $\check{R}^{V_\mu, V_\mu}$ acts on $V_\nu \subseteq V_\mu \otimes V_\mu$ by the scalar 
$$\pm q^{{1\over 2}\chi_\nu(C) - \chi_\mu(C)}.$$ 
\Cref{casimir eigenvalues} defines the eigenvalue $\chi_\gamma(C) = \langle \gamma, \gamma + 2\rho \rangle$ of the Casimir operator of $\Ug$ acting on the irreducible $V_\gamma$. 

We obtain the \textit{braided exterior algebra} $\bigwedge_q(V)$ as the quotient of the tensor algebra $\mathcal{T}(V)$ modulo the ideal generated by the \textit{positive} eigenvectors of $\check{R}$ \cite[Remark~2.3]{berenstein}. That is, we set
\begin{align}\label{braided ext alg as tensor alg quotient}
	\textstyle \bigwedge_q(V) \coloneqq \mathcal{T}(V) / \big\langle\mathrm{Sym}^2_q(V) \big\rangle,
\end{align}
with $\mathrm{Sym}^2_q(V)$ denoting the set of eigenvectors of $\check{R}$ with eigenvalue of the form $+q^r$ for some $r \in \mathbb{Q}$. \Cref{braided ext alg as tensor alg quotient} is the quantum analogue of \Cref{exterior algebra def}, which constructs the classical exterior algebra using the $+1$ eigenvectors of $\tau$, the classical enveloping algebra's (symmetric) braiding operator. In fact, a continuity argument shows that the positive eigenvectors of $\check{R}$ become elements of $\mathrm{Sym}^2(V) = \ker\left(\tau + \id_{V \otimes V}\right)$ in the limit as $q \to 1$ \cite[Corollary~2.22(3)]{leduc_ram_1997}.

We now specialize $\lieg = \mathfrak{gl}_n$ and focus on the braided exterior algebra $\bigwedge_q(\Vn)$, with $\Vn$ denoting the natural $\Uqgln$-module. We begin by defining $\Uqgln$ and computing its action on $\bigwedge_q(\Vn)$.

\begin{defn}\label[defn]{uqgln def}
	The \textit{general linear quantum group} $\Uqgln$ is the unital associative algebra generated by $\Uqsln$ along with the pairwise commuting group-like generators $L_i^{\pm 1}$, for $i = 1, \ldots, n$ satisfying the following relations:
	\begin{align*}
	\begin{gathered}
		K_i = L_i \inv{L_{i+1}}, \quad
		L_i E_j \inv{L_i} = q^{\delta_{ij}} E_j, 
		\quad \text{and} \quad
		L_i F_j \inv{L_i} = q^{-\delta_{ij}} F_j.
	\end{gathered}
	\end{align*}
\end{defn}
\begin{rmk}
	The algebra $\Uqgln$ is known as the \textit{semisimple} form of the quantum group associated to the Cartan type $A_{n-1}$, since for each element $\mu = \sum \mu_i \omega_i$ in the $\mathfrak{sl}_n$ weight lattice there is $K_\mu = \prod_{i=1}^n L_i^{\mu_i}$ in $\Uqgln$ \cite[Remark~9.1.1]{chari_pressley_1994}. This is consistent with the following observation. 
	
	Note that if $\epsilon_i$ denotes the standard basis $\mathfrak{h}^*(\mathfrak{sl}_n)$, which is orthonormal with respect to inner product $\langle \cdot, \cdot \rangle$ defined by \Cref{inner prod on roots}, then $\alpha_i = \epsilon_i - \epsilon_{i+1}$ is the $i$th simple positive root of $\mathfrak{sl}_n$ and the relations $L_i E_j \inv{L_i} = q^{\delta_{ij}}$ and $L_i F_j \inv{L_i} = q^{-\delta_{ij}} F_j$ are equivalent to 
$$
	L_i E_j \inv{L_i} = q^{\langle \epsilon_i, \alpha_j \rangle} E_j 
	\quad \text{and} \quad
	L_i F_j \inv{L_i} = q^{-\langle \epsilon_i, \alpha_j \rangle} F_j.
$$
\end{rmk}

In this chapter we consider $\Uqgln$ as a co-algebra with the comultiplication described in \Cref{delta convention}. Although usually merely a convention, the choice of comultiplication map is important in \Cref{commuting quantum actions section}, where we must use different conventions for each quantum group to ensure that $\Uqgln$ and $\Uqglm$ induce commuting subalgebras of module endomorphisms.

The natural $\Uqgln$-module $\Vn$ has a weight basis $v_i$, for $i = 1, \ldots, n$, so that
\begin{align*}
	K_i \rhd v_j = q^{\langle \alpha_i, \epsilon_j \rangle} v_j.
\end{align*}
Since the natural $\gln$-module is minuscule, it lifts to the quantum group $\Uqgln$ in such a way that the matrices of \Cref{glv matrices} describing the Chevalley generators remain the same \cite[Section~5A.1]{J}. The following explicit formulae define the natural $\Uqgln$-module:
\begin{align}\label{natural uqgln action}
	\begin{gathered}
		E_i \rhd v_j = \delta_{i+1, j} v_i, \quad 
		F_i \rhd v_j = \delta_{ij} v_{i+1}, 
		\quad \text{and} \quad
		L_i \rhd v_j = q^{\delta_{ij}} v_j.
	\end{gathered}
\end{align}

Now we compute $\mathrm{Sym}_q^2(\Vn)$ by studying the decomposition of $\Vn \otimes \Vn$ as a $\Uqgln$-module. We decompose $\Vn \otimes \Vn$ by finding \textit{highest weight vectors}, that is, simultaneous eigenvectors of the $K_i$ that are annihilated by every $E_i$. 
\begin{lem}\label[lem]{nat uqgln sq decomp}
	Let $\Vn$ denote the natural $\Uqgln$-module, as in \Cref{natural uqgln action}. As a $\Uqgln$-module, $\Vn \otimes \Vn$ decomposes as 
	$$\Vn \otimes \Vn \cong V_{\epsilon_1 + \epsilon_2} \oplus V_{2 \epsilon_1}$$
	The irreducible module $V_{\epsilon_1 + \epsilon_2}$ is spanned by
	$$v_i \otimes  v_j - \qinv v_j \otimes v_i, \quad i < j,$$
	and $V_{2\epsilon_1}$ is spanned by
	$$v_i \otimes v_i, \quad v_i \otimes  v_j + q v_j \otimes v_i, \quad i < j$$
\end{lem}
\begin{proof}
	Using the formulae above, is easy to check that $v_1 \otimes v_2 - \qinv v_2 \otimes v_1$ and $v_1 \otimes v_1$ are weight vectors of weights $\epsilon_1 + \epsilon_2$ and $2\epsilon_1$. In addition, a straightforward computation shows they are annihilated by all the $E_i$. The irreducible modules $V_{\epsilon_1 + \epsilon_2}$ and $V_{2\epsilon_1}$ are cyclic, so the spanning property follows from an inductive argument using the formulas of \Cref{natural uqgln action}. The $U(\mathfrak{gl}_n)$- and $\Uqgln$-modules parametrized by the same dominant weight have the same dimension, so the decomposition follows.
\end{proof}

Notice that the vectors spanning $V_{2 \epsilon_1}$ become elements of $\mathrm{Sym}^2(V)$ in the limit $q \to 1$. This means that $V_{2 \epsilon_1} \subset \Vn \otimes \Vn$ contains all positive eigenvectors of $\check{R}$ and therefore $\bigwedge_q(\Vn)$ is the quotient 
\begin{align*}
	\mathcal{T}(\Vn) / \big\langle \mathrm{Sym}^2_q(\Vn) \big\rangle 
	= \mathcal{T}(\Vn) / \langle v_i \otimes v_i, v_i \otimes v_j + q v_j \otimes v_i \mid i < j \rangle
\end{align*}
when $\Vn$ is the natural $\Uqgln$-module. Equivalently, $\bigwedge_q(\Vn)$ the unital associative algebra generated by $v_1, \ldots, v_n$ subject to the following relations:
\begin{align}\label{braided ext alg rels}
	\begin{split}
		v_i^2 &= 0, \quad i = 1, \ldots, n\\
		v_i v_j &= -q v_j v_i, \quad i < j.
	\end{split}
\end{align}
In particular, the set of
\begin{align}\label{braided ext alg basis}
	v(\ell) = \begin{cases}
		v_1^{\ell_1} v_2^{\ell_2} \cdots v_n^{\ell_n}, & \text{if } \ell_i \in \{0, 1\} \\
		0, & \text{otherwise},
	\end{cases}
\end{align}
for $\ell \in \{0, 1\}^n$, forms a basis of $\bigwedge_q(\Vn)$ \cite[Lemma~2.32]{berenstein}. This basis closely resembles the basis of $\bigwedge(\mathbb{C}^n)$ described by \Cref{bar v ext alg basis review}. In fact, $\bigwedge_q(\Vn) \cong \bigwedge(\mathbb{C}^n)$ as vector spaces but \textit{not} as algebras.

Since $\Uqgln$ acts on the tensor algebra $\mathcal{T}(\Vn)$ via the comultiplication map, it preserves the underlying algebraic structure of $\bigwedge_q(\Vn)$. In other words, $\bigwedge_q(\Vn)$ is a $\Uqgln$-\textit{module algebra}.

\begin{defn}\cite[Definition~4.1.1]{montgomery_1993}\label[defn]{module alg defn}
	An associative algebra $(A, \mu)$ with multiplication $\mu$ is a $\Uqg$-\textit{module algebra} if $A$ is a $\Uqg$-module such that the $\Uqg$-action preserves the algebraic structure of $A$, meaning that
	$$X \rhd \mu(a \otimes b) = \mu(\Delta(X) (a\otimes b)), \quad \text{and} \quad X \rhd 1 = \epsilon(X) 1,$$
	for all $a, b \in A$ and $X \in \Uqg$.
\end{defn}

We are now ready to show that the $\Uqgln$-action on $\bigwedge_q(\Vn)$ factors through the \textit{quantum Clifford algebra} $Cl_q(n)$, so that \Cref{uqgln clq embedding compatibility diag} commutes. Recall \Cref{clqnk defn}. In this section, we fix $k = 1$, use the shorthand $Cl_q(n)$ in place of $Cl_q(n, 1)$, and only consider the module $\Omega^0$ defined in \Cref{clqnk repn}. For convenience, we review the $Cl_q(n)$-action on $\bigwedge_q(\Vn)$:
\begin{align}\label{clqn action by raising lowering}
	\begin{split}
		\psi_j \rhd v(\ell) &= (-1)^{\ell_1 + \cdots + \ell_{j-1}} v(\ell - e_j), \\
		\psi_j^\dagger \rhd v(\ell) &= (-1)^{\ell_1 + \cdots + \ell_{j-1}} v(\ell + e_j), \quad \text{and}\\
		\omega_j \rhd v(\ell) &= q^{-\ell_j} v(\ell).
	\end{split}
\end{align}

The next proposition constructs the map $\Phi_{q, n}\colon \Uqgln \to Cl_q(n)$ illustrated in \Cref{uqgln clq embedding compatibility diag} by factoring the $\Uqgln$-action on $\bigwedge_q(\Vn)$ into products of $Cl_q(n)$ operators, much like \Cref{glv action by inner and exterior mult} factors the $\mathfrak{gl}(V)$-action in the classical case into products of the inner and exterior multiplication operators of \Cref{inner ext mult ops gln review}. This result is a quantum analogue of \Cref{gln cl embedding}. In this case we consider the \textit{quantized} inner and exterior multiplication operators $\iota_i^q$ and $\varepsilon_i^q$ satisfying
\begin{align}\label{quantized inner ext mult ops gln review}
\begin{split}
\iota_i^q(v(\ell)) &= (-q)^{\ell_1 + \cdots + \ell_{i-1}} v(\ell - e_i)
	= \pi_0 \bigg(\prod_{k < i} \omega_k^{-1} \psi_i \bigg) (v(\ell))
	\quad \text{and}\\
	\varepsilon_i^q(v(\ell)) &= (-\qinv)^{\ell_1 + \cdots + \ell_{i-1}} v(\ell + e_i)
	= \pi_0 \bigg(\prod_{k < i} \omega_k \psi_i^\dagger \bigg) (v(\ell))
\end{split}
\end{align}
for any $v(\ell) \in \bigwedge_q(\Vn)$\longer{, as in \Cref{q inner and ext mult ops}}. These operators consider the underlying algebra structure of $\bigwedge_q(\Vn)$ and include correction terms to correctly account for powers of $(-q)$ that result from the exchange of $v_i$'s needed to express the result of an inner or exterior multiplication in terms of the $v(\ell)$ basis defined by \Cref{braided ext alg basis}.

\begin{prop}\label[prop]{uqgln clqn embedding}
Let $\iota^q_i$ and $\varepsilon^q_i$ denote the quantized inner and exterior multiplication operators defined by \Cref{quantized inner ext mult ops gln review}. There is an algebra map $\Phi_{q, n}\colon \Uqgln \to Cl_q(n)$ satisfying
\begin{align}	
	E_i \to \varepsilon_i^q \iota_{i+1}^q 
	\qquad \text{and} \qquad 
	F_i \to \varepsilon_{i+1}^q \iota_{i}^q,
\end{align}
for $i = 1, \ldots, n-1$. In terms of $Cl_q(n)$ generators, $\Phi_{q, n}$ maps
\begin{align}
	E_i \to \qinv \omega_i^{-1} \psi_i^\dagger \psi_{i+1}^{\pdg}, 
	\quad 
	F_i \to \omega_i \psi_{i+1}^\dagger \psi_i^{\pdg}, 
	\quad \text{and} \quad 
	L_i \to \omega_{i}^{-1}.
\end{align}
\end{prop}
\begin{rmk}
	Recall \Cref{inner ext operator action review} and \Cref{gln cl embedding}: in the classical case the root vectors $E_i$ and $F_i$ of $\gln$ map to $\varepsilon_i \iota_{i+1}$ and $\varepsilon_{i+1} \iota_i$ in $Cl\left(\mathbb{C}^n \oplus (\mathbb{C}^n)^*\right)$.
\end{rmk}
\begin{proof}
	The claims follow from direct calculation. It suffices to show that the images $\widetilde{E}_i, \widetilde{F}_i$, and $\widetilde{K}_i$ of $E_i, F_i$, and $K_i$ under $\Phi_{q, n}$ satisfy the relations defining $\Uqgln$. Most relations are easy consequences of the relations defining $Cl_q(n)$ and \crossrefcliff{identities in clq}, which is a generalized analogue of \cite[Lemma~3.1]{hayashi_1990}. For instance, we may easily obtain
	\begin{align*}
		[\widetilde{E}_i, \widetilde{F}_i] 
		= [\psi_{i+1}^{\pdg}\psi_i^\dagger, \psi_i^{\pdg} \psi_{i+1}^\dagger] 
		= \frac{(\omega_i \omega_{i+1}^{-1}) - (\omega_i \omega_{i+1}^{-1})^{-1}}{q - \qinv} 
		= \frac{\widetilde{K}_i - \widetilde{K}_i^{-1}}{q - \qinv}.
	\end{align*}
	When $|i - j| > 1$, the quantum Serre relations follow immediately from the anticommutation relations $\psi_i \psi_i^\dagger + q^{\pm 1} \psi_i^\dagger \psi_i = \omega_i^{\mp 1}$. We use \Cref{alt uq Serre rels} to verify the remaining quantum Serre relations. Applying \crossrefcliff{q comm 4 psi}, we find that
	\begin{align*}
		[\tilde{E_i}, \widetilde{E}_{i+1}]_q &= (\psi_i^\dagger \psi_{i+1}^{\pdg}) (\psi_{i+1}^\dagger \psi_{i+2}^{\pdg}) - q (\psi_{i+1}^\dagger \psi_{i+2}^{\pdg})(\psi_i^\dagger \psi_{i+1}^{\pdg}) 
		= \omega_{i+1}^{-1} \psi_i^\dagger \psi_{i+2}^{\pdg},
	\end{align*}
	so
	\begin{align*}
		[\tilde{E_i}, [\tilde{E_i}, \widetilde{E}_{i+1}]_q]_{\qinv} 
		= -\omega_{i+1}^{-1} \big(q \psi_{i+1}^{\pdg}(\psi_i^\dagger)^2 \psi_{i+2}^{\pdg}- \qinv \psi_{i+2} (\psi_i^\dagger)^2 \psi_{i+1}^{\pdg}\big) 
		= 0.
	\end{align*}
	Since $\Phi_{q, n}(F_i) = \psi_{i+1}^\dagger \psi_i^{\pdg}(q \omega_i^{-1}) = \Phi_{q, n}(E_i)^*$, it follows that the $\Phi_{q, n}(F_i)$ also satisfy the $q$-Serre relations.
\end{proof}

The next proposition proves that the $\Uqgln$-action on $\bigwedge_q(\Vn)$ induced by $\Phi_{q, n}$ coincides with the $\Uqgln$-module algebra action.

\begin{prop}\label[prop]{uqgln action factors through clqn}
	The $\Uqgln$-action on its module algebra $\bigwedge_q(\Vn)$ factors through the quantum Clifford algebra $Cl_q(n)$. Concretely, \Cref{uqgln clq embedding compatibility diag} commutes.
\end{prop}
\begin{proof}
	Since $\bigwedge_q(\Vn)$ is a $\Uqgln$-module algebra, it follows that
	\begin{align*}
		E_i \rhd v(\ell) &= (E_i v_1^{\ell_1}) (K_i v_2^{\ell_2}) (K_i v_3^{\ell_3}) \cdots (K_i v_n^{\ell_n}) 
		 + v_1^{\ell_1} (E_i v_2^{\ell_2}) (K_i v_3^{\ell_3}) \cdots (K_i v_n^{\ell_n}) \\
		&\quad + \cdots + v_1^{\ell_1} v_2^{\ell_2} \cdots v_{n-1}^{\ell_{n-1}}(E_i v_n^{\ell_n}) \\
		&= \sum_{j=1}^n 
			\delta_{i+1,j} \cdot \ell_j \cdot q^{\sum_{k > j} \langle \alpha_i, \ell_k \epsilon_k \rangle} 
			(v_1^{\ell_1} \cdots v_{j-1}^{\ell_{j-1}} 
				\cdot v_i \cdot 
			v_{j+1}^{\ell_{j+1}} \cdots v_n^{\ell_n})
		\\
		&= (-q)^{\ell_1 + \cdots \ell_{i-1}} v_i 
			\cdot \ell_{i+1}
			(v_1^{\ell_1} \cdots v_{i}^{\ell_{i}} v_{i+2}^{\ell_{i+2}} \cdots v_n^{\ell_n} ) \\
		&=  v_i \cdot \iota_{i+1}^q v(\ell) \\
		&= \varepsilon_i^q \iota_{i+1}^q v(\ell) \\
		&= \Phi_{q, n}(E_i) v(\ell).
	\end{align*}
	In the second-to-last equality we use the commutation relations \eqref{braided ext alg rels} defining the braided exterior algebra $\bigwedge_q(\Vn)$. A similar calculation shows that $F_i \rhd v(\ell) = \Phi_{q, n}(F_i) v(\ell)$.
	

	Finally, notice that
	\begin{align*}
		L_i \rhd v(\ell) &= (L_i v_1^{\ell_1}) \cdots (L_i v_n^{\ell_n}) 
		= q^{\ell_i} v(\ell) 
		= \omega_i^{-1} v(\ell) 
		= \Phi_{q, n}(L_i) v(\ell).
		&\qedhere
	\end{align*}
\end{proof}

The map $\Phi_{q, n}$ immediately yields a multiplicity-free decomposition of the $\Uqgln$-module algebra $\bigwedge_q(\Vn)$. The irreducible $\Uqgln$-modules in $\bigwedge(\Vn)$ are parametrized by the same dominant weights as in the corresponding classical result.
\begin{prop}\label[prop]{uqgln ext alg decomp}
	As a $\Uqgln$-module, 
	$$\textstyle \bigwedge_q(\Vn) \cong \displaystyle \bigoplus_{j = 0}^n V_{\gamma_j}.$$
	Here $V_\mu$ denotes the irreducible $\Uqgln$-module with highest weight $\mu$, and each $\gamma_j = \sum_{k=1}^j e_k$ denotes a fundamental weight of $\gln$.
\end{prop}
\begin{proof}
	Let $\alpha_i = e_i - e_{i+1}$. Notice that $v(\gamma_j)$ is a $\Uqgln$-highest weight vector, since it is annihilated by each $E_i$:
	\begin{align*}
	E_i \rhd v(\gamma_j)
		= \psi_i^\dagger \psi_{i+1}^{\pdg} v(\gamma_j)
		= v(\gamma_j + \alpha_i)
		= 0.
	\end{align*}
	The claim now follows by a dimension count, considering the classical decomposition of $\gln$-modules $\bigwedge(V) \cong \bigoplus_{j = 0}^n \bigwedge^j(V)$.
\end{proof}


	\subsection{Commuting actions of $\Uqgln$ and $\Uqglm$ on $\bigwedge_q(V^{(nm)})$}
	\label{commuting quantum actions section}

Fix $n$ and $m$. Recall \Cref{gln glm embeddings factor through glnm cd}, which illustrates the construction of commuting embeddings $\lambda$ and $\rho$ of $\gln$ and $\glm$ into the Clifford algebra $\Clnm$ as a composition of known Lie algebra maps. These embeddings are critical to our proof of the skew $GL_n \times GL_m$-duality \Cref{skew gln glm duality} in the classical case.

In this subsection, we mimic that construction, or at least part of it, to obtain the quantum analogues $\lambda_q$ and $\rho_q$ mapping $\Uqgln$ and $\Uqglm$ into $Cl_q(nm)$. Much like their classical counterparts, the maps $\lambda_q$ and $\rho_q$ play a prominent role in the proof of our skew $\Uqgln \otimes \Uqglm$-duality \Cref{uqgln uqglm duality}: \Cref{uqgln uqglm embeddings commute} proves that they define commuting subalgebras of $\End(\bigwedge_q(V^{(nm)})$, which we use to compute the multiplicity-free decomposition of $\bigwedge_q(V^{(nm)})$ in \Cref{type a q mf decomposition}.

For the rest of this section, we focus on the braided exterior algebra $\bigwedge_q(V^{(nm)})$, as defined by \Cref{braided ext alg as tensor alg quotient}, with $V^{(nm)}$ denoting the natural $\Uqglnm$-module. Recall that  $V^{(nm)}$ has a $\Uqglnm$-weight basis $v_i$, for $i = 1, \ldots, nm$. In addition, recall that $\bigwedge_q(V^{(nm)})$ is generated as an algebra by the same $v_i$, subject to the relations in \eqref{braided ext alg rels}. In particular, $\bigwedge_q(V^{(nm)})$ has a basis $v(\ell)$, for $\ell \in \{0, 1\}^{nm}$, as described by \Cref{braided ext alg basis}. In this case, we say that each basis element $v(\ell)$ describes a unique state of \textit{occupied} and \textit{vacant} positions in an $n \times m$ grid.

To begin, we obtain $\lambda_q\colon \Uqgln \to Cl_q(nm)$ as the composition resulting from the following diagram, which is the quantum analogue of the top half of \Cref{gln glm embeddings factor through glnm cd}:
\begin{equation}\label[diag]{uqgln embedding into clqnm diag}
	\begin{tikzcd}[column sep={3.6cm,between origins}, row sep={1.8cm,between origins}]
		\Uqgln \ar[r, "\Delta^{(m-1)}"]  \ar[d, swap, "\Theta"] 
 		\ar[drr, dashed, blue!60, "\lambda_q"]
 		& \Uqgln^{\otimes m} \ar[r, "\Phi_{q, n}^{\otimes m}"]
 		& Cl_q(n)^{\otimes m} \ar[d, "\, \Gamma_q"] \rar
 		& \End\left(\bigwedge_q(\Vn)^{\otimes m}\right) \arrow[d, "\,\, \cong", ] \\
		\Uqglnm \ar[rr, swap, "\Phi_{q, nm}"]
 		&
 		& Cl_q(nm) \rar
 		& \End\left(\bigwedge_q(V^{(nm)})\right) 
 	\end{tikzcd}
\end{equation}
Recall that \Cref{clqn clqnm embedding} defines the isomorphism $\Gamma_q$.

On one hand, there is a $\Uqgln$-action on $\bigwedge_q(V^{(nm)})$ that factors through $\Uqglnm$. Motivated by the embedding $\gln \to \glnm$ of \Cref{gln embedding into glnm}, we construct an algebra map $\Theta\colon \Uqgln \to \Uqglnm$ in \Cref{uqgln uqglnm embedding}. Recall \Cref{natural uqgln action}, which describes the action of the $\Uqglnm$ generators on $V^{(nm)}$. As in the classical case, the $\Uqgln$ root vectors ``align'' with those of $\Uqglnm$, as depicted in \Cref{root vector action on basis}. Composing $\Theta$ with the $Cl_q(nm)$-representation $\Phi_{q, nm}$ defined in \Cref{uqgln clqn embedding} results in an algebra map $\Uqgln \to \End\left(\bigwedge_q(V^{(nm)})\right)$. 

On the other hand, the braided exterior algebra $\bigwedge_q(\Vn)$ is a $\Uqgln$-module, so $\Uqgln$ acts on $\bigwedge(\Vn)^{\otimes m}$ via comultiplication. 

\Cref{isom tensor power braided ext alg} proves that $\bigwedge_q(\Vn)^{\otimes m} \cong \bigwedge_q(V^{(nm)})$ as $\Uqgln$-modules, so the two $\Uqgln$-actions on $\bigwedge_q(V^{(nm)})$ in fact coincide.

The next proposition defines the map $\Theta\colon \Uqgln \to \Uqglnm$ illustrated in \Cref{uqgln embedding into clqnm diag} as a first step in defining a $\Uqgln$-action on $\bigwedge_q(V^{(nm)})$. It is the quantum analogue of \Cref{gln embedding into glnm}. To the best of our knowledge, this map is new.

\begin{prop}\label[prop]{uqgln uqglnm embedding}
	Recall the superscript notation of \Cref{generator superscript}: $X_j^{(p)}$ denotes a $U_q(\mathfrak{gl}_p)$ generator. Define
	\begin{align*}
		\Lambda_{i, <j} = \prod_{p < j} K_{i + (p-1)n}^{(nm)},
		\quad \text{and} \quad
		\Lambda_{i, >j} = \prod_{p > j} K_{i + (p-1)n}^{(nm)},
	\end{align*}
	by taking the product of all $K_p^{(nm)}$ generators on the $i$th row to the left, and respectively to the right, of the $j$th column.

	There is an algebra homomorphism $\Theta\colon \Uqgln \rightarrow \Uqglnm$ satisfying
	\begin{align*}
		E_i^{(n)} &\to \sum_{j=1}^m E_{i + (j-1)n}^{(nm)} \Lambda_{i, >j} \\
		F_i^{(n)} &\to \sum_{j=1}^m \Lambda_{i, <j}^{-1} F_{i + (j-1)n}^{(nm)}  \\
		L_i^{(n)} &\to \prod_{j=1}^m L_{i + (j-1)n}^{(nm)}.
	\end{align*}
\end{prop}
\begin{proof}
	It suffices to show that the images $\widetilde{E}_i, \widetilde{F}_i$ and $\widetilde{K}_i$ of $E_i^{(n)}, F_i^{(n)}$ and $K_i^{(n)}$ under $\Theta$ satisfy the relations of \Cref{uqgln def}. This follows from the definition of $\Uqglnm$ directly from a calculation. 
	
	Recall notation of \Cref{inner prod on roots}. In what follows $a_{ij}$ denotes the $(i,j)$ entry of the Cartan matrix of $\gln$, while each $\alpha_k$, for $k = 1, \ldots, nm$, denotes a simple positive root of $U_q(\mathfrak{gl}_{nm})$. In addition, the form $\langle \cdot, \cdot \rangle$ denotes an inner product on $\mathfrak{h}_{\mathfrak{gl}_{nm}}^* \times \mathfrak{h}_{\mathfrak{gl}_{nm}}^*$. To begin, note that
	\begin{align*}
		\widetilde{K}_j \widetilde{E}_i \inv{\widetilde{K}_j} 
		&= \sum_{b=1}^m q^{\langle \alpha_{i + (b-1)n}, \alpha_{j + (b-1)n} \rangle} E_{i + (b-1)n}^{(nm)} \Lambda_{i, >b} 
		= q^{2\delta_{ij} - \delta_{i,j+1} - \delta_{i+1, j}} \widetilde{E}_i 
		= q^{a_{ij}} \widetilde{E}_i.
	\end{align*}
	A similar calculation shows that $\widetilde{K}_j \widetilde{F}_i \inv{\widetilde{K}_j} = q^{-a_{ij}} \widetilde{F}_i$.

	Next observe that
	\begin{align*}
		[\widetilde{E}_i, \widetilde{F}_j] 
		&= \sum_{a, b} q^{\langle \alpha_{i + (a-1)n}, \sum_{c < b} \alpha_{j + (c-1)n} \rangle} \Lambda_{j, <b}^{-1}\Lambda_{i, >a}  E_{i + (a-1)n}^{(nm)} F_{j + (b-1)n}^{(nm)} \\
		&\quad - \sum_{a, b} q^{\langle \alpha_{j + (b-1)n}, \sum_{c > a} \alpha_{i + (c-1)n} \rangle} \Lambda_{j, <b}^{-1}\Lambda_{i, >a} F_{j + (b-1)n}^{(nm)} E_{i + (a-1)n}^{(nm)} \\
		&= \sum_a \Lambda_{j, <a}^{-1}\Lambda_{i, >a}  [E_{i + (a-1)n}^{(nm)}, F_{j + (a-1)n}^{(nm)}] \\
		&= {\delta_{ij} \over q - \qinv} \bigg(\sum_a \prod_{c < a} \big(K_{i + (c-1)n}^{(nm)}\big)^{-1} \cdot \prod_{c \geq a} K_{i + (c-1)n}^{(nm)} \\
		&\qquad - \sum_a \prod_{c \leq a} \big(K_{i + (c-1)n}^{(nm)}\big)^{-1}  \cdot \prod_{c > a} K_{i + (c-1)n}^{(nm)}\bigg) \\
		&= {\delta_{ij} \over q - \qinv} \bigg(\prod_{c = 1}^m K_{i + (c-1)n}^{(nm)} - \prod_{c = 1}^m \big(K_{i + (c-1)n}^{(nm)}\big)^{-1}\bigg) \\
		&= \delta_{ij} \frac{\widetilde{K}_i - \inv{\widetilde{K}_i}}{q- \qinv}.
	\end{align*}
	In the second equality we use the commutators $[E_i^{(nm)}, F_j^{(nm)}] = 0$ for $i \neq j$.

	Finally, we verify the Serre relations. Recall \ref{not and conv} and let $a_{ij}$ denote the entries of the Cartan matrix for $\gln$, as in \Cref{cartan mats}. For any $1 \leq i, j \leq n-1$, the inner product $\langle \alpha_{i+(a-1)n}, \alpha_{j+(b-1)n} \rangle$ vanishes
	whenever $a \neq b$, since then $|(i + (a-1)n) - (j + (b-1)n)| \geq 2$. Therefore, 
	$$[E_{i + (a-1)n}^{(nm)}, E_{j + (b-1)n}^{(nm)}] = 0$$
	for any $1 \leq i, j \leq n-1$ whenever $a \neq b$. In other words, simple positive root vectors $E_{i + (a-1)n}^{(nm)}$ and $E_{j + (b-1)n}^{(nm)}$ acting on different columns always commute. First suppose $|i - j| > 1$. In this case $[E_{i + (a-1)n}^{(nm)}, E_{j + (a-1)n}^{(nm)}] = 0$ for each $a = 1, \ldots, m$, so 
	\begin{align*}
		[\widetilde{E}_i, \widetilde{E}_j] &= \sum_{a, b} q^{-\langle \alpha_{i + (a-1)n}, \sum_{c > b} \alpha_{j + (c-1)n} \rangle} \Lambda_{j, >b}\Lambda_{i, >a}  E_{i + (a-1)n}^{(nm)} E_{j + (b-1)n}^{(nm)} \\
		&\quad - \sum_{a, b} q^{-\langle \alpha_{j + (b-1)n}, \sum_{c > a} \alpha_{i + (c-1)n} \rangle} \Lambda_{j, >b}\Lambda_{i, >a} E_{j + (b-1)n}^{(nm)} E_{i + (a-1)n}^{(nm)} \\
		&= \sum_a \Lambda_{j, >a}\Lambda_{i, >a} [E_{i + (a-1)n}^{(nm)}, E_{j + (a-1)n}^{(nm)}] \\
		&= 0.
	\end{align*}

	Now suppose $j = i+1$. With \Cref{alt uq Serre rels} in mind, we first compute
	\begin{align*}
		[\widetilde{E}_i, \widetilde{E}_{j}]_q 
		&= \sum_{a < b} \big(1 - q^{1 - a_{ji}}\big) \Lambda_{j, >b}\Lambda_{i, >a} E_{i + (a-1)n}^{(nm)} E_{j + (b-1)n}^{(nm)} \\
		&\quad + \sum_a \Lambda_{j, >a}\Lambda_{i, >a} [E_{i + (a-1)n}^{(nm)}, E_{j + (a-1)n}^{(nm)}]_q \\
		&= (1 - q^2)\sum_{a < b} \Lambda_{j, >b}\Lambda_{i, >a} E_{i + (a-1)n}^{(nm)} E_{j + (b-1)n}^{(nm)} \\
		&\quad + \sum_a \Lambda_{j, >a}\Lambda_{i, >a} [E_{i + (a-1)n}^{(nm)}, E_{j + (a-1)n}^{(nm)}]_q. \\
		&= (1 - q^2) I + II.
	\end{align*}
	Next, we calculate
	\begin{align*}
		[\widetilde{E}_i, II]_{\qinv} &= \sum_{a, b} q^{- \langle \alpha_{i + (a-1)n}, \sum_{c > b} \alpha_{i + (c-1)n} + \alpha_{j + (c-1)n} \rangle} \Lambda_{j, >b}\Lambda_{i, >b} \Lambda_{i, >a} \\
		&\quad \cdot E_{i + (a-1)n}^{(nm)} [E_{i + (b-1)n}^{(nm)}, E_{j + (b-1)n}^{(nm)}]_q \\
		&\quad - \qinv \sum_{a, b} q^{-\langle \alpha_{i + (b-1)n} + \alpha_{j + (b-1)n}, \sum_{c > a} \alpha_{i + (c-1)n} \rangle} \Lambda_{j, >b}\Lambda_{i, >b} \Lambda_{i, >a} \\
		&\quad \cdot [E_{i + (b-1)n}^{(nm)}, E_{j + (b-1)n}^{(nm)}]_q E_{i + (a-1)n}^{(nm)} \\
		&= \sum_{a < b} \big(1 - q^{-1 - a_{ii} - a_{ji}} \big)\Lambda_{j, >b}\Lambda_{i, >b} \Lambda_{i, >a} E_{i + (a-1)n}^{(nm)} [E_{i + (b-1)n}^{(nm)}, E_{j + (b-1)n}^{(nm)}]_q \\
		&\quad + \sum_a \Lambda_{j, >a}\Lambda_{i, >a}^2 [E_{i + (a-1)n}^{(nm)}, [E_{i + (b-1)n}^{(nm)}, E_{j + (b-1)n}^{(nm)}]_q]_{\qinv} \\
		&= (1 - q^{-2})\sum_{a < b} \Lambda_{j, >b}\Lambda_{i, >b} \Lambda_{i, >a} E_{i + (a-1)n}^{(nm)} [E_{i + (b-1)n}^{(nm)}, E_{j + (b-1)n}^{(nm)}]_q 
		\\
		&= (1 - q^{-2}) III.
	\end{align*}
	In the second equality we used the $q$-Serre relations defining $\Uqglnm$, which guarantee that $[E_{i + (b-1)n}^{(nm)}, E_{j + (b-1)n}^{(nm)}]_q]_{\qinv} = 0$. Continuing, we see that
	\begin{align*}
		[\widetilde{E}_i, I]_{\qinv} &= \sum_{c, a< b} q^{- \langle \alpha_{i + (c-1)n}, \sum_{s > a} \alpha_{i + (s-1)n} + \sum_{s > b} \alpha_{j + (s-1)n} \rangle} \Lambda_{j, >b} \Lambda_{i, >a} \Lambda_{i, c^+} \\
		&\quad \cdot E_{i + (c-1)n}^{(nm)} E_{i + (a-1)n}^{(nm)} E_{j + (b-1)n}^{(nm)} \\
		&\quad  - \qinv \sum_{c, a < b} q^{-\langle \alpha_{i + (a-1)n} + \alpha_{j + (b-1)n}, \sum_{s > a} \alpha_{i + (s-1)n} \rangle} \\
		&\quad \cdot E_{i + (a-1)n}^{(nm)} E_{j + (b-1)n}^{(nm)} E_{i + (c-1)n}^{(nm)} \\
		&= \sum_{c < a < b} \big(1 - q^{-1 - a_{ii} - a_{ji}}\big) \Lambda_{j, >b} \Lambda_{i, >a} \Lambda_{i, c^+} E_{i + (c-1)n}^{(nm)} E_{i + (a-1)n}^{(nm)} E_{j + (b-1)n}^{(nm)} \\
		&\quad + \sum_{a < c < b} \big(q^{-a_{ii}} - q^{-1 - a_{ji}}\big) \Lambda_{j, >b} \Lambda_{i, c^+} \Lambda_{i, >a}  E_{i + (a-1)n}^{(nm)} E_{i + (c-1)n}^{(nm)} E_{j + (b-1)n}^{(nm)} \\
		&\quad + \sum_{a < b < c} \big(q^{-a_{ii} - a_{ij}} - \qinv\big) \Lambda_{j, >b} \Lambda_{i, >a} \Lambda_{i, c^+} E_{i + (c-1)n}^{(nm)} E_{i + (a-1)n}^{(nm)} E_{j + (b-1)n}^{(nm)} \\
		&\quad + \sum_{c = a < b} \big(1 - q^{-1 - a_{ji}}\big) \Lambda_{j, >b} \Lambda_{i, >a}^2 (E_{i + (a-1)n}^{(nm)})^2 E_{j + (b-1)n}^{(nm)} \\
		&\quad + \sum_{a < b = c} \Lambda_{j, >b} \Lambda_{i, >a} \Lambda_{i, >b} \big(q^{-a_{ii}} E_{i + (a-1)n}^{(nm)} E_{i + (b-1)n}^{(nm)} E_{j + (b-1)n}^{(nm)} \\
		&\qquad \qquad - \qinv E_{i + (a-1)n}^{(nm)} E_{j + (b-1)n}^{(nm)}E_{i + (b-1)n}^{(nm)}\big).
	\end{align*}
	In the second equality, the second summation is the negative of the first so their sum vanishes; each term in the third and fourth summations is trivial so they too vanish. The last summation may be re-written using $q$-commutators:
	\begin{align*}
		\sum_{a < b} \Lambda_{j, >b} &\Lambda_{i, >a} \Lambda_{i, >b} \big(q^{-a_{ii}} E_{i + (a-1)n}^{(nm)} E_{i + (b-1)n}^{(nm)} E_{j + (b-1)n}^{(nm)} - \qinv E_{i + (a-1)n}^{(nm)} E_{j + (b-1)n}^{(nm)}E_{i + (b-1)n}^{(nm)}\big) \\
		&= q^{-2}\sum_{a < b} \Lambda_{j, >b} \Lambda_{i, >a} \Lambda_{i, >b} E_{i + (a-1)n}^{(nm)} [E_{i + (b-1)n}^{(nm)}, E_{j + (b-1)n}^{(nm)}]_q. \\
		&= q^{-2} III.
	\end{align*}
	Hence, $[\widetilde{E}_i, I]_{\qinv} = q^{-2} III$. Combining results, we conclude that
	\begin{align*}
		[\widetilde{E}_i, [\widetilde{E}_i, \widetilde{E}_j]_q]_{\qinv} 
		= (1 - q^2)[\widetilde{E}_i, I]_{\qinv} + [\widetilde{E}_i, II]_{\qinv} 
		= 0,
	\end{align*}
	as desired. We verify that the $\widetilde{F_i}$ also satisfy the $q$-Serre relations using a similar calculation, which we omit here.
\end{proof}

We obtain an algebra map $\lambda_q\colon \Uqgln \to Cl_q(nm)$ by composing $\Theta$ with $\Phi_{q, nm}$, as in \Cref{uqgln embedding into clqnm diag}.

\begin{prop}\label[prop]{uqgln clqnm embedding}
	Define
	$$\kappa_{i, <j} = \prod_{p < j} \omega_{i + (p-1)n}^{-1} \omega_{i+1 + (p-1)n} \quad \text{and} \quad \kappa_{i, > j} = \prod_{p > j} \omega_{i + (p-1)n}^{-1} \omega_{i+1 + (p-1)n},$$
	by taking an appropriate product of $\omega_a$ generators in the $i$th and $(i+1)$st rows to the left, and respectively to the right, of the $j$th column.

	There is an algebra homomorphism $\lambda_q\colon \Uqgln \to Cl_q(nm)$ satisfying
	\begin{align*}
		E_i^{(n)} &\to \qinv \sum_{j=1}^m \omega_{i + (j-1)n}^{-1} \psi^\dagger_{i + (j-1)n} \psi_{i+1 + (j-1)n} \kappa_{i, >j}, \\
		F_i^{(n)} &\to \sum_{j=1}^m \omega_{i + (j-1)n} \kappa_{i, <j}^{-1} \psi^\dagger_{i+1 + (j-1)n} \psi_{i + (j-1)n}^{\pdg}, \text{ and} \\
		L_i^{(n)} &\to \prod_{j=1}^m \omega_{i + (j-1)n}^{-1}.
	\end{align*}
	We use superscripts on $\Uqgln$ generators as in \Cref{generator superscript}.
\end{prop}
\begin{rmk}
	Notice that when $m = 1$, $\lambda_q$ is exactly the homomorphism $\Phi_{q, n}$ of \Cref{uqgln clqn embedding}.
\end{rmk}
\begin{proof}
	The map $\lambda_q$ is simply the composition $\Phi_{q, n} \circ \Theta$, so the statement follows from \Cref{uqgln clqn embedding} and \Cref{uqgln uqglnm embedding}.
\end{proof}

The map $\Theta\colon \Uqgln \to \Uqglnm$ of \Cref{uqgln uqglnm embedding} equips the $\Uqglnm$-module $\bigwedge_q(V^{(nm)})$ with a $\Uqgln$-module structure. Composing the map $\Delta^{(m-1)} \circ \Phi_{q, n}$ with the isomorphism $\Gamma_q\colon Cl_q(n)^{\otimes m} \to Cl_q(nm)$ gives $\bigwedge_q(V^{(nm)})$ another $\Uqglnm$-module structure. The next proposition proves the two actions coincide.

\begin{prop}\label[prop]{isom tensor power braided ext alg}
	\Cref{uqgln embedding into clqnm diag} commutes. In particular, there is an isomorphism of $\Uqgln$-modules
	\begin{align*}
		\textstyle \bigwedge_q(V^{(nm)}) \cong \bigwedge_q(\Vn)^{\otimes m}.
	\end{align*}
\end{prop}
\begin{proof}
	Using the definitions, it is straightforward to check that $\lambda_q = \Gamma_q \circ \Delta^{(m-1)} \circ \Phi_{q, n}$. Combining with the commutativity of \Cref{braided ext alg isom as qcl modules} proves the claim.
\end{proof}
\begin{rmk}
	The isomorphism of $\Uqgln$-modules can be promoted to an isomorphism of $\Uqgln$-\textit{module algebras} if we deform the multiplication in $\bigwedge_q(\Vn)^{\otimes m}$ as in Theorem~2.3 of \cite{lzz_2010}.
\end{rmk}

We now turn to studying the commutant of the $\Uqgln$-action on $\bigwedge_q(V^{(nm)})$. We begin by defining a $\Uqglm$-module structure on $\bigwedge_q(V^{(nm)})$ that factors through $Cl_q(nm)$ in \Cref{uqglm clqnm embedding}. Our skew Howe duality \Cref{uqgln uqglm duality} proves that this action generates the centralizer algebra $\End_{\Uqgln}\left(\bigwedge_q(V^{(nm)})\right)$.

Unlike its classical counterpart $\rho$ of \Cref{glm embedding into clnm}, the map $\rho_q$ does \textit{not} factor through $\Uqglnm$. That is, there is no quantum analogue of \Cref{glm embedding into glnm}. This has to do with our choice of weight basis; see \Cref{column major order means glm action does not factor}. Two issues arise in attempting to quantize \Cref{glm embedding into glnm}. First, the classical $\glm$ generators $E_j^{(m)}$ map to non-simple $\glnm$ root vectors, implemented by highly nested commutators of $E_a^{(nm)}$ generators. Although root vectors corresponding to non-simple roots have a quantum analogue in terms of $q$-commutators as defined in Section~7.3.1 of \cite{klimyk_schmudgen_1997}, the resulting expressions are rather complicated. Second, the commutation relations amongst images of $\glm$ generators rely on the Jacobi identity, which can be quantized in many different ways.

Notwithstanding, our map $\rho_q\colon \Uqglm \to Cl_q(nm)$ makes the following diagram commute:
\begin{equation}\label[diag]{uqglm embedding into clqnm diag}
	\begin{tikzcd}[column sep={4.5cm,between origins}, row sep={1.8cm,between origins}]
 		&
 		& Cl_q(nm) \\
 		\Uqglm \ar[r, swap, "\widetilde{\Delta}^{(n-1)}"] 
 		\ar[urr, dashed, blue!60, "\rho_q", bend left=13]
 		& \Uqglm^{\otimes n} \ar[r, swap, "\Phi_{q, m}^{\otimes n}"]
 		& Cl_q(m)^{\otimes n} \ar[u, swap, "\, \Gamma_q"]
 	\end{tikzcd}
\end{equation}
Note that \Cref{uqglm embedding into clqnm diag} quantizes part of \Cref{gln glm embeddings factor through glnm cd}.
\begin{rmk}
	In this \chorpaper\ the choice of comultiplication is important when discussing commuting quantum group actions. In order to ensure that the $\Uqglm$-action on $\bigwedge_q(V^{(nm)})$ commutes with the $\Uqgln$-action defined by $\lambda_q$, we must use the comultiplication $\widetilde{\Delta}\colon \Uqglm \to \Uqglm^{\otimes 2}$ satisfying
	\begin{align}\label{backward comult}
		\begin{split}
			\widetilde{\Delta}(E_i) &= E_i \otimes 1 + K_i \otimes E_i \\
			\widetilde{\Delta}(F_i) &= F_i \otimes \kinv + 1 \otimes F_i \\
			\widetilde{\Delta}(L_i) &= L_i \otimes L_i.
		\end{split}
	\end{align}
	This comultiplication \textit{does not} agree with that  defined by $\Delta$ as in \Cref{delta convention}. In $\widetilde{\Delta}(E_i)$, for instance, $K_i$ appears in the \textit{first} tensor factor. It is also possible to swap the comultiplications for $\Uqgln$ and $\Uqglm$. While choosing the same convention for both factors leads to well-defined embeddings into $Cl_q(nm)$, the resulting actions on $\bigwedge_q(V^{(nm)})$ do \textit{not} commute. In the setting of orthogonal algebras discussed in \crossrefsonch, the formulas embedding a commuting factor of $\Uqprime$ into $\Uqglm$ depend on a choice of comultiplication for $\Uqson$.
\end{rmk}
\begin{prop}\label[prop]{uqglm clqnm embedding}
	Define
	$$\kappa_{<i, j} = \prod_{p < i} \omega_{p + (j-1)n}^{-1} \omega_{p + jn} \quad \text{and} \quad \kappa_{> i, j} = \prod_{p > i} \omega_{p + (j-1)n}^{-1} \omega_{p + jn}$$
	by taking an appropriate product of $\omega_a$ generators in the $j$th and $(j+1)$st columns above, and respectively below, of the $i$th row.

	There is an algebra homomorphism $\rho_q\colon \Uqglm \to Cl_q(nm)$ mapping
	\begin{align*}
		E_j^{(m)} &\to \sum_{i=1}^n \kappa_{<i, j} \psi^{\dagger}_{i + (j-1)n} \psi_{i + jn}^{\pdg}, \\
		F_j^{(m)} &\to \sum_{i=1}^n \psi^{\dagger}_{i + jn} \psi_{i + (j-1)n}^{\pdg}\kappa_{>i, j}^{-1}, \text{ and} \\
		L_j^{(m)} &\to \prod_{i=1}^n \omega_{i + (j-1)n}^{-1}.
	\end{align*}
	As usual, we use superscript indices as in \Cref{generator superscript}. 
\end{prop}
\begin{proof}
	It suffices to show that the images $\widetilde{E}_j, \widetilde{F}_j$, and $\widetilde{K}_j$ of $E_j^{(m)}, F_j^{(m)}$, and $K_j^{(m)}$ under $\rho_q$ satisfy the relations defining $\Uqglm$. The following calculations are very similar to those proving \Cref{uqgln uqglnm embedding}. For starters, notice that
	\begin{align*}
		\tilde{L}_j \widetilde{E}_\ell \tilde{L}_j^{-1} 
		&= \sum_a q^{\delta_{j \ell}} \kappa_{<i, j} \psi_{a + (\ell - 1)n}^\dagger \psi_{a + \ell n}
		= q^{\delta_{j \ell}} \widetilde{E}_j 
	\end{align*}
	Next we verify that
	$$[\widetilde{E}_j, \widetilde{F}_\ell] = \delta_{j \ell} {\widetilde{K}_j - \widetilde{K}_j^{-1} \over q - \qinv}$$
	by summing over different summation index regimes independently. Notice that when $a < b$, the product $\kappa_{>b, \ell}$ commutes with $\psi_{a + (j-1)n}^\dagger \psi_{a + jn}$. Similarly, $\kappa_{<a, j}$ commutes with $\psi_{b + \ell n}^\dagger \psi_{b + (\ell - 1)n}$ when $a < b$. Therefore, the double sum
	\begin{align*}
		[\widetilde{E}_j, \widetilde{F}_\ell] &= \sum_{a, b} \kappa_{<a, j} \psi_{a + (j-1)n}^\dagger \psi_{a + jn}^{\pdg}
		 \kappa_{>b, \ell}^{-1} \psi_{b + \ell n}^\dagger \psi_{b + (\ell - 1)n}^{\pdg} \\
		&\quad - \sum_{a, b} \kappa_{>b, \ell}^{-1} \psi_{b + \ell n}^\dagger \psi_{b + (\ell - 1)n}^{\pdg}
		 \kappa_{<a, j}\psi_{a + (j-1)n}^\dagger \psi_{a + jn}^{\pdg}
	\end{align*}
	vanishes in the regime $a < b$. Hence
	\begin{align*}
		[\widetilde{E}_j, \widetilde{F}_\ell] &= \sum_{a > b} \big(q^{a_{\ell j}} - q^{a_{j \ell}} \big)
		\kappa_{<a, j}^{\phantom{-1}} \kappa_{>b, \ell}^{-1} \, \psi_{a + (j-1)n}^\dagger \psi_{a + jn}^{\pdg}  \psi_{b + \ell n}^\dagger \psi_{b + (\ell - 1)n}^{\pdg} \\
		&\quad + \sum_a \kappa_{<a, j} \kappa_{>a, \ell}^{-1} \, [\psi_{a + (j-1)n}^\dagger \psi_{a + jn}^{\pdg}, \psi_{a + \ell n}^\dagger \psi_{a + (\ell - 1)n}^{\pdg}] \\
		&= \frac{\delta_{j \ell}}{q - \qinv} \sum_a \kappa_{<a, j} \kappa_{>a, \ell}^{-1} \, \big( (\omega_{a + (j-1)n}^{-1} \omega_{a + jn}) - (\omega_{a + (j-1)n}^{-1} \omega_{a + jn})^{-1}\big) \\
		&= \frac{\delta_{j \ell}}{q - \qinv} \bigg(\prod_{i \leq n} \omega_{i + (j-1)n}^{-1} \omega_{i + jn} - \prod_{i \geq 1} \big(\omega_{i + (j-1)n}^{-1} \omega_{i + jn}\big)^{-1}\bigg) \\
		&= \delta_{j \ell} \frac{\widetilde{K}_j - \widetilde{K}_j^{-1}}{q - \qinv}.
	\end{align*}
	Note that the sum in the second equality is telescoping.

	Finally, we check that $\widetilde{E}_j$ and $\widetilde{F}_j$ satisfy the Serre relations. First suppose that $|j - \ell| > 1$. In this case,
	\begin{align*}
		[\widetilde{E}_j, \widetilde{E}_\ell] &= \sum_{a < b} \big(q^{-a_{\ell j}} - 1 \big)
			\kappa_{<a, j}\kappa_{<b, \ell} \, \psi_{a + (j-1)n}^\dagger \psi_{a + jn}^{\pdg} \psi_{b + (\ell - 1)n}^\dagger \psi_{b + \ell n}^{\pdg}\\
			&\quad + \sum_{a > b} \big(1 - q^{-a_{j \ell}}\big)
			\kappa_{<a, j}\kappa_{<b, \ell} \, \psi_{a + (j-1)n}^\dagger \psi_{a + jn}^{\pdg}\psi_{b + (\ell - 1)n}^\dagger \psi_{b + \ell n}^{\pdg} \\
			&\quad + \sum_a \kappa_{<a, j}\kappa_{<a, \ell} \, \psi_{a + (j-1)n}^\dagger \psi_{a + jn}^{\pdg} \psi_{a + (\ell - 1)n}^\dagger \psi_{a + \ell n}^{\pdg} \\
			&\quad - \sum_a \kappa_{<a, j}\kappa_{<a, \ell} \, \psi_{a + (\ell - 1)n}^\dagger \psi_{a + \ell n} \psi_{a + (j-1)n}^\dagger \psi_{a + jn}^{\pdg}\\
			&= 0.
	\end{align*}
	The first and second sums vanish because $a_{j \ell} = 0$ when $|j - \ell| > 1$. The creation and annihilation operators in the last summation may be arranged in the order specified by the third sum with no overall change in sign. This means the fourth sum is the negative of the third.

	Now suppose $\ell = j + 1$, so that $a_{j \ell} = -1$. In this case we compute
	\begin{align*}
		[\widetilde{E}_j, \widetilde{E}_\ell]_q 
		&= \sum_{a > b} \big(1 - q^{1 -a_{\ell j}} \big)
		\kappa_{<a, j}\kappa_{<b, \ell} \, \psi_{a + (j-1)n}^\dagger \psi_{a + jn}^{\pdg} \psi_{b + (\ell - 1)n}^\dagger \psi_{b + \ell n}^{\pdg}\\
		&\quad + \sum_a \kappa_{<a, j}\kappa_{<a, \ell} \, [\psi_{a + (j-1)n}^\dagger \psi_{a + jn}^{\pdg},
		\psi_{a + (\ell - 1)n}^\dagger \psi_{a + \ell n}^{\pdg}]_q \\
		&= (1 - q^2) \sum_{a > b} \kappa_{<a, j}\kappa_{<b, \ell} \,
			\psi_{a + (j-1)n}^\dagger \psi_{a + jn}^{\pdg} \psi_{b + (\ell - 1)n}^\dagger \psi_{b + \ell n}^{\pdg} \\
		&\quad + \sum_a \kappa_{<a, j}\kappa_{<a, \ell} \, \omega_{a+jn}^{-1} \psi_{a+(j-1)n}^\dagger \psi_{a + \ell n}^{\pdg}\\
		&= (1 - q^2) I + II.
	\end{align*}
	In the third equality we used \crossrefcliff{q comm 4 psi} with $\varphi_i = \psi_{a + (j-1)n}^\dagger$ and $\varphi_k = \psi_{a+ \ell n}$.

	Further, observe that
	\begin{align*}
		[\widetilde{E}_j, II]_{\qinv} 
			&= \sum_{a > b} \big(1 - q^{-1-\langle \alpha_j, \varepsilon_j \rangle} \big)
				\kappa_{<a, j} \kappa_{<b, j} \kappa_{<b, \ell} \, \omega_{b + jn}^{-1} \,
				\psi_{a + (j-1)n}^\dagger \psi_{a + jn}^{\pdg} \psi_{b + (j-1)n}^\dagger \psi_{b + \ell n}^{\pdg} \\
			&\quad + q \sum_a \kappa_{<a,j}^2 \kappa_{<a, \ell} \, \omega_{a + jn}^{-1} \,
				\psi_{a + (j-1)n}^\dagger \psi_{a+jn}^{\pdg}\psi_{a + (j-1)n}^\dagger \psi_{a + \ell n}^{\pdg}\\
			&\quad - \qinv \sum_a \kappa_{<a,j}^2 \kappa_{<a, \ell} \, \omega_{a + jn}^{-1} \,
				\psi_{a + (j-1)n}^\dagger \psi_{a + \ell n} \psi_{a + (j-1)n}^\dagger \psi_{a+jn} \\
			&= (1 -  q^{-2}) \sum_{a > b} \kappa_{<a, j} \kappa_{<b, j} \kappa_{<b, \ell} \, \omega_{b + jn}^{-1} \,
				\psi_{a + (j-1)n}^\dagger \psi_{a + jn} \psi_{b + (j-1)n}^\dagger \psi_{b + \ell n} \\
			&\quad - q \sum_a \kappa_{<a,j}^2 \kappa_{<a, \ell} \, \omega_{a + jn}^{-1} \,
				\psi_{a+jn}^{\pdg} \big(\psi_{a + (j-1)n}^\dagger\big)^2 \psi_{a + \ell n}^{\pdg} \\
			&\quad + \qinv \sum_a \kappa_{<a,j}^2 \kappa_{<a, \ell} \, \omega_{a + jn}^{-1} \,
				\psi_{a+\ell n}^{\pdg} \big(\psi_{a + (j-1)n}^\dagger\big)^2 \psi_{a + j n}^{\pdg} \\
			&= (1 -  q^{-2}) \sum_{a > b} \kappa_{<a, j} \kappa_{<b, j} \kappa_{<b, \ell} \, \omega_{b + jn}^{-1} \,
				\psi_{a + (j-1)n}^\dagger \psi_{a + jn}^{\pdg}\psi_{b + (j-1)n}^\dagger \psi_{b + \ell n}^{\pdg}\\
			&= (1 - q^{-2}) III.
	\end{align*}
	In addition, we see that
	\begin{align*}
		[\widetilde{E}_j, I]_{\qinv} &= \sum_{c < b < a} \big(q^{-a_{jj} - a_{\ell j}} - \qinv\big)
				\kappa_{<a, j} \kappa_{<b, \ell} \kappa_{<c, j} \, \\
				&\quad \quad \quad \quad \cdot
				\psi_{c + (j-1)n}^\dagger \psi_{c + jn}^{\pdg}
				\psi_{a + (j-1)n}^\dagger \psi_{a +jn}^{\pdg}
				\psi_{b + (\ell - 1)n}^\dagger \psi_{b + \ell n}^{\pdg} \\
			&\quad + \sum_{b < c < a} \big(q^{-a_{jj}} - q^{-1 - a_{j \ell}}\big)
				\kappa_{<a, j} \kappa_{<b, \ell} \kappa_{<c, j} \, \\
				&\quad \quad \quad \quad \cdot
				\psi_{c + (j-1)n}^\dagger \psi_{c + jn}^{\pdg}
				\psi_{a + (j-1)n}^\dagger \psi_{a +jn}^{\pdg}
				\psi_{b + (\ell - 1)n}^\dagger \psi_{b + \ell n}^{\pdg}\\
			&\quad + \sum_{b < a < c} \big(1 - q^{-1 - a_{j \ell} - a_{jj}}\big)
				\kappa_{<c, j} \kappa_{<b, \ell} \kappa_{<a, j} \, \\
				&\quad \quad \quad \quad \cdot
				\psi_{a + (j-1)n}^\dagger \psi_{a + jn}^{\pdg}
				\psi_{c + (j-1)n}^\dagger \psi_{c +jn}^{\pdg}
				\psi_{b + (\ell - 1)n}^\dagger \psi_{b + \ell n}^{\pdg} \\
			&\quad + \sum_{c = b < a} \kappa_{<a, j} \kappa_{<b, \ell} \kappa_{<b, j} \, \\
				&\quad \quad \quad \quad \cdot \big(q^{-a_{jj}} \,
				\psi_{b + (j-1)n}^\dagger \psi_{b + jn}^{\pdg}
				\psi_{a + (j-1)n}^\dagger \psi_{a + jn}^{\pdg}
				\psi_{b + (\ell - 1)n}^\dagger \psi_{b + \ell n}^{\pdg} \\
				&\quad \quad \quad \quad - \qinv \,
				\psi_{a + (j-1)n}^\dagger \psi_{a + jn}^{\pdg}
				\psi_{b + (\ell - 1)n}^\dagger \psi_{b + \ell n}^{\pdg}
				\psi_{b + (j-1)n}^\dagger \psi_{b + jn}^{\pdg}\big) \\
			&\quad + \sum_{b < a = c} \kappa_{<a, j}^2 \kappa_{<b, \ell}
				\big(\psi_{a+(j-1)n}^\dagger\big)^2 \big(\psi_{a+jn}^{\pdg}\big)^2 \psi_{b+(\ell-1)n}^\dagger \psi_{b+\ell n}^{\pdg}\\
			&\quad - \sum_{b < a = c} q^{-1 - a_{j \ell}} \kappa_{<a, j}^2 \kappa_{<b, \ell}
				\big(\psi_{a+(j-1)n}^\dagger\big)^2 \big(\psi_{a+jn}^{\pdg}\big)^2 \psi_{b+(\ell-1)n}^\dagger \psi_{b+\ell n}^{\pdg} \\
			&= q^{-2} \sum_{a > b} \kappa_{<a, j} \kappa_{<b, \ell} \kappa_{<b, j} \,
				\psi_{a +(j-1)n}^\dagger \psi_{a+jn}^{\pdg}
				[\psi_{b + (j-1)n}^\dagger \psi_{b+jn}^{\pdg}, \psi_{b +(\ell-1)n}^\dagger \psi_{b+\ell n}^{\pdg}]_q \\
			&= q^{-2} \sum_{a > b} \kappa_{<a, j} \kappa_{<b, \ell} \kappa_{<b, j} \, \omega_{b + jn}^{-1} \,
				\psi_{a +(j-1)n}^\dagger \psi_{a+jn}^{\pdg} \psi_{b + (j-1)n}^\dagger \psi_{b + \ell n}^{\pdg} \\
			&= q^{-2} III.
	\end{align*}
	In the third equality we used \crossrefcliff{q comm 4 psi} with $\varphi_i = \psi_{b + (j-1)n}^\dagger$ and $\varphi_k = \psi_{b + \ell n}$. Combining results, and using identity \eqref{alt uq Serre rels}, we conclude that
	\begin{align*}
		[\widetilde{E}_j, [\widetilde{E}_j , \widetilde{E}_\ell]_q]_{\qinv} 
			= (1 - q^2)q^{-2} III + (1-q^{-2})III 
			= 0,
	\end{align*}
	as required. Since $\widetilde{F}_j = \rho_q(E_j)^*$, we may apply the map $*\colon Cl_q(nm) \to Cl_q(nm)$ of \crossrefcliff{star struct} to conclude that the $\widetilde{F}_j$ also satisfy the quantum Serre relations.
\end{proof}

Now we take a moment to interpret the actions of $\Uqgln$ and $\Uqglm$ on $\bigwedge_q(V^{(nm)})$ defined by \Cref{uqgln clqnm embedding,uqgln clqnm embedding}. Recall the explanatory remarks below \Cref{glm embedding into clnm}. As in the classical case, it is convenient to arrange the $\Uqglnm$-weight vectors $v_{i + (j-1)n}$ of $V^{(nm)}$, for $i = 1, \ldots, n$ and $j = 1, \ldots, m$, in an $n \times m$ rectangular array like \eqref{basis rect diagram}. The $Cl_q(nm)$ generators $\psi_{i+(j-1)n}$ and $\psi_{i + (j-1)n}^\dagger$ act on a given state vector $v(\ell)$ by attempting to vacate and occupy the $(i, j)$ position. Attempting to vacate an unoccupied position kills the state, as does attempting to occupy a filled position. Each $\omega_{a}$ merely scales $v(\ell)$ by $q^{-\ell_a}$, so the $\kappa$ factors in the definition of $\lambda_q$ act diagonally on state vectors to account for the action of $K_i^{(n)}$ generators in $\Delta^{(m-1)}(E_i^{(n)})$ and $\Delta^{(m-1)}(F_i^{(n)})$. Analogous comments apply to the $\kappa$ in the definition of $\rho_q$. Therefore, if we let $p_j$ and $p_i'$ denote some integers, and we set $\alpha_{i,j}^v = e_{i,j} - e_{i+1, j}$ and $\alpha_{i,j}^h = e_{i,j} - e_{i, j+1}$, we see, for instance, that
\begin{align}
	\lambda_q(E_i^{(n)}) \, v(\ell) 
	&= \sum_{j=1}^m (-q)^{p_j} v(\ell + \alpha_{i,j}^v) 
	\label{uqgln ei on basis}
	\quad \text{and} \\
	\rho_q(E_j^{(m)}) \, v(\ell) 
	&= \sum_{i=1}^n (-q)^{p_i'} v(\ell + \alpha_{i,j}^h)
	\label{uqglm ej on basis}
\end{align}
for any state vector $v(\ell)$. That is, the $i$th positive root vector $E_i^{(n)}$ in $\Uqgln$ attempts to shift each occupied position in the $(i+1)$st row \textit{upwards}, while the $j$th positive root vector $E_j^{(m)}$ in $\Uqglm$ attempts to shift occupied positions in the $(j+1)$st column \textit{to the left}. This means that \Cref{action of ein,action of ejm} still illustrate the action of $\Uqgln$ and $\Uqglm$ root vectors induced by $\lambda_q$ and $\rho_q$, if we replace the signs by suitable powers of $-q$. This means the $\Uqgln$ and $\Uqglm$ root vectors act as their classical counterparts in the limit $q \to 1$: compare \Cref{uqgln ei on basis,uqglm ej on basis} to \Cref{gln ei on basis,glm ej on basis}.

In addition, observe that each $L_i^{(n)}$ and each $L_j^{(m)}$ induces a \textit{quantized degree operator} that is the exponential of a degree operator acting as in \Cref{diag L action}. For instance, the action of $\lambda_q(L_i^{(n)})$ counts occupied positions in the $i$th \textit{row}:
$$
	L_i^{(n)} \rhd v(\ell) 
	= q^{\sum_{j=1}^m \ell_{i + (j-1)m}} 
	= q^{\lambda(\bar{L}_i)} v(\ell).
$$
Recall \Cref{gln embedding into clnm} defines $\lambda\colon \gln \to Cl\left(\mathbb{C}^n \oplus (\mathbb{C}^n)^*\right)$ in the classical case. Dually, the action of $\rho_q(L_j^{(m)})$ counts occupied positions in the $j$th \textit{column}. The action of root vectors in $\Uqgln$ and $\Uqglm$ preserves spaces of homogeneous column and row degree, so it must commute with the action of the $L_j^{(m)}$ and the $L_i^{(n)}$, respectively.

The next proposition shows that although $\lambda_q\colon \Uqgln \to Cl_q(nm)$ and $\rho_q\colon \Uqglm \to Cl_q(nm)$ do \textit{not} define commuting subalgebras of $Cl_q(nm)$, they do induce commuting actions on $\bigwedge_q(V^{nm})$. Recall that in the classical case the maps $\lambda\colon \gln \to \Clnm$ and $\rho\colon \glm \to \Clnm$ of \Cref{gln embedding into clnm,glm embedding into clnm} define commuting subalgebras of $\Clnm \cong \End(\bigwedge(\mathbb{C}^{nm}))$. \Cref{clqnk is semisimple} implies that quantum representation $Cl_q(nm) \to \End(\bigwedge_q(V^{(nm)})$ is \textit{not} faithful. This means, for instance, that the \textit{actions} of $\lambda_q(E_i^{(n)})$ and $\rho_q(E_j^{(m)})$ may commute even if $[\lambda_q(E_i^{(n)}), \, \rho_q(E_j^{(m)})]$ is \textit{not} the zero element in $Cl_q(nm)$.

\begin{prop}\label[prop]{uqgln uqglm embeddings commute}
	The embeddings $\lambda_q$ and $\rho_q$ of \Cref{uqgln clqnm embedding} and \Cref{uqglm clqnm embedding} define commuting subalgebras of $\End\left(\bigwedge_q(V^{(nm)})\right)$.
\end{prop}
\begin{proof}
	The proof follows from direct calculation. Considering the explanatory remarks preceding the proposition, it remains to show that the action induced by the root vectors of $\Uqgln$ commutes with the action induced by the root vectors of $\Uqglm$. We begin by computing commutators in the quantum Clifford algebra. As usual, our strategy is to consider different summation index regimes independently. For instance, observe that when $a > j + 1$ or $b < i$, the product $\kappa_{<b, j}$ commutes with $\psi_{i + (a-1)n}^\dagger \psi_{i+1 + (a-1)n}$ and similarly the product $\kappa_{i, >a}$ commutes with $\psi_{b + (j-1)n}^\dagger \psi_{b + jn}$. Therefore, the sum
	\begin{align*}
		[\lambda_q(E_i^{(n)}), \rho_q(E_j^{(m)})] &=
			\sum_{a = 1}^m \sum_{b = 1}^n \big( \kappa_{i, >a} \, \psi_{i + (a-1)n}^\dagger \psi_{i+1+(a-1)n}^{\pdg}
				\kappa_{<b, j} \, \psi_{b + (j-1)n}^\dagger \psi_{b + jn}^{\pdg} \\
			&\quad\quad\quad - \kappa_{<b, j} \, \psi_{b + (j-1)n}^\dagger \psi_{b + jn}^{\pdg}
				\kappa_{i, >a} \, \psi_{i + (a-1)n}^\dagger \psi_{i+1+(a-1)n}^{\pdg}  \big)
	\end{align*}
		vanishes in the regime where $b < i$ or $a > j+1$ because we can appropriately rearrange the creation and annihilation operators. Thus we see that the commutator
	\begin{align*}
		[\lambda_q(E_i^{(n)}), \rho_q(E_j^{(m)})]
			&= \sum_{a, b = i, i+1} \kappa_{i, >a} \, \psi_{i + (a-1)n}^\dagger \psi_{i+1+(a-1)n}^{\pdg}
				\kappa_{<b, j} \, \psi_{b + (j-1)n}^\dagger \psi_{b + jn}^{\pdg} \\
			&\quad - \sum_{a, b = i, i+1} \kappa_{<b, j} \, \psi_{b + (j-1)n}^\dagger \psi_{b + jn}^{\pdg}
				\kappa_{i, >a} \, \psi_{i + (a-1)n}^\dagger \psi_{i+1+(a-1)n}^{\pdg} \\
			&= \sum_{\substack{a = j, j + 1 \\ b = i, i + 1}} [\kappa_{i, >a}
				\psi_{i + (a-1)n}^\dagger \psi_{i+1 + (a-1)n}^{\pdg},
				\kappa_{<b, j}
				\psi_{b + (j-1)n}^\dagger \psi_{b + jn}^{\pdg}].
	\end{align*}
	We now take a closer look at the four remaining commutators. When $(a, b) = (j, i)$, each term in the commutator
	$[\kappa_{i, >a} \psi_{i + (a-1)n}^\dagger \psi_{i+1 + (a-1)n}, \kappa_{<b, j} \psi_{b + (j-1)n}^\dagger \psi_{b + jn}]$ contains a factor of $\big(\psi_{i + (j-1)n}^\dagger \big)^2$, so it vanishes. Similarly, when $(a, b) = (j+1, i+1)$, each term contains a factor of $\big(\psi_{i+1 + jn}\big)^2$, so it also vanishes. Only two terms remain:
	\begin{align*}
		[\lambda_q(E_i^{(n)}), \rho_q(E_j^{(m)})] &=
				\qinv \kappa_{i, >j} \kappa_{<(i+1), j}
				[\psi_{i+(j-1)n}^\dagger \psi_{i+1 +(j-1)n}^{\pdg}, \psi_{i+1 + (j-1)n}^\dagger \psi_{i+1 + jn}^{\pdg}] \\
			&\quad + \kappa_{i, >(j+1)} \kappa_{<i, j}
				[\psi_{i + jn}^\dagger \psi_{i+1 +jn}^{\pdg}, \psi_{i + (j-1)n}^\dagger \psi_{i + jn}^{\pdg}] \\
			&= \qinv \kappa_{i, >j} \kappa_{<(i+1), j} \psi_{i + (j-1)n}^\dagger
				\big\{ \psi_{i+1 +(j-1)n}^{\pdg}, \psi_{i+1 + (j-1)n}^\dagger \big\}
				\psi_{i+1 + jn}^{\pdg} \\
			&\quad - \kappa_{i, >(j+1)} \kappa_{<i, j} \psi_{i + (j-1)n}^\dagger
				\big\{ \psi_{i +jn}^\dagger, \psi_{i +jn}^{\pdg}\big\}
				\psi_{i+1 + jn}^{\pdg}\\
			&= \kappa_{i, >(j+1)} \kappa_{<i, j} \big( \qinv \omega_{i + (j-1)n}^{-1} \omega_{i+1 + jn}
				\big\{ \psi_{i+1 +(j-1)n}^{\pdg}, \psi_{i+1 + (j-1)n}^\dagger \big\} \\
			&\quad - \big\{ \psi_{i +jn}^\dagger, \psi_{i +jn}^{\pdg} \big\} \big)
				\psi_{i + (j-1)n}^\dagger \psi_{i+1 + jn}^{\pdg}.
	\end{align*}
	In the second-to-last equality we used \Cref{psi psid comm anticomm} to slide the anticommutators to the left of $\psi_{i+(j-1)n}^\dagger$.
	
	Using the last equality, we see that the commutator $[\lambda_q(E_i^{(n)}), \rho_q(E_j^{(m)})]$ annihilates $\bigwedge_q(V^{(nm)})$. Indeed, \Cref{clqn action by raising lowering} implies each anticommutator $\{\psi_a, \psi_a^\dagger\}$ acts as the identity map, so for any basis vector $v(\ell)$ of $\bigwedge_q(V^{(nm)})$,
	\begin{align*}
		[\lambda_q(E_i^{(n)}), \rho_q(E_j^{(m)})] \, v(\ell)
			&= (-1)^{\sum_{p = i+1 + (j-1)n}^{i + jn} p} \kappa_{i, >(j+1)} \,\kappa_{<i, j} \\
			&\quad \quad  \cdot \big( \qinv \omega_{i + (j-1)n}^{-1} \omega_{i+1 + jn}
				\big\{ \psi_{i+1 +(j-1)n}^{\pdg}, \psi_{i+1 + (j-1)n}^\dagger \big\} \\
			&\quad \quad - \big\{ \psi_{i +jn}^\dagger, \psi_{i +jn}^{\pdg}\big\} \big) v(\ell + e_{i + (j-1)n} - e_{i+1 + jn}) \\
			&= (-1)^{\sum_{p = i+1 + (j-1)n}^{i + jn} p} \, \kappa_{i, >(j+1)} \kappa_{<i, j} \\
			&\quad \quad \cdot \big( v(\ell + e_{i + (j-1)n} - e_{i+1 + jn}) - v(\ell + e_{i + (j-1)n} - e_{i+1 + jn}) \big) \\
			&= 0.
	\end{align*}
	Similar calculations show that $[\lambda_q(E_i^{(n)}), \rho_q(F_j^{(m)})]$ also induces the zero map. Applying the $*$-operation defined by \Cref{star struct} shows that the remaining commutators also induce the zero map, since
\begin{align*}
	[\lambda_q(E_i^{(n)}), \rho_q(E_j^{(m)})]^*
		&= -[\lambda_q(F_i^{(n)}), \rho_q(F_j^{(m)})],
		\quad \text{and} \\[+0.3em]
	[\lambda_q(E_i^{(n)}), \rho_q(F_j^{(m)})]^*
		&= -[\lambda_q(E_i^{(n)}), \rho_q(F_j^{(m)})].
	&\qedhere
\end{align*}
\end{proof}


	\subsection{Multiplicity-free decomposition of $\bigwedge_q(V^{(nm)})$}
	\label{type a q mf decomposition}

In the last section we found homomorphic images of $\Uqgln$ and $\Uqglm$ in the quantum Clifford algebra $Cl_q(nm)$ that generate commuting actions on the $Cl_q(nm)$-module $\bigwedge_q(V^{(nm)})$. In this subsection we use those actions to compute a multiplicity-free decomposition of $\bigwedge_q(V^{(nm)})$ as a $\Uqgln \otimes \Uqglm$-module. We rely on the classical skew duality \Cref{skew gln glm duality} and on the Double Commutant \Cref{double comm}(b) to prove our main quantized duality result, \Cref{uqgln uqglm duality}.

\begin{thm}\label{uqgln uqglm duality}
	As a $\Uqgln \otimes \Uqglm$-module, 
	\begin{align}\label{uqgln uqglm duality decomp}
		\textstyle \bigwedge_q(V^{(nm)}) = \displaystyle \bigoplus_\mu V_\mu^{(n)} \otimes V^{(m)}_{\mu'}.
	\end{align}
	The sum ranges over all partitions $\mu$ fitting in an $n \times m$ rectangle, $V_\mu^{(p)}$ denotes the irreducible $U_q(\mathfrak{gl}_p)$-module parametrized by $\mu$, and $\mu'$ denotes the conjugate of $\mu$. In particular, $\bigwedge_q(V^{(nm)})$ is multiplicity-free as a $\Uqgln \otimes \Uqglm$-module. 
	
	Consequently, the maps $\lambda_q \colon \Uqgln \to Cl_q(nm)$ and $\rho_q \colon \Uqglm \to Cl_q(nm)$ generate mutual commutants in $\End\big(\bigwedge_q(V^{(nm)})\big)$.
\end{thm}
\begin{rmk}
	Recall that in the classical case, we first establish that $\gln$ and $\glm$ generate each others commutants in $\End\left(\bigwedge(\mathbb{C}^n \otimes \mathbb{C}^m)\right)$, and \textit{then} we use the double commutant yoga of \Cref{double comm}(a) to conclude that the decomposition of $\bigwedge(\mathbb{C}^n \otimes \mathbb{C}^m)$ as a $U(\gln) \otimes U(\glm)$-module is multiplicity-free. 
	
	In the quantum case, the argument flows the opposite way. First we construct a multiplicity-free decomposition of $\bigwedge_q(V^{(nm)})$ as a $\Uqgln \otimes \Uqglm$-module and \textit{then} we use the Double Commutant \Cref{double comm}(b) to conclude that $\Uqgln$ and $\Uqglm$ indeed generate each others full commutants in $\End\left(\bigwedge_q(V^{(nm)}) \right)$. Our proof of \Cref{uqgln uqglm duality} requires the classical result and a dimension count to guarantee that the decomposition of $\bigwedge_q(V^{(nm)})$ is in fact multiplicity-free.
\end{rmk}

\Cref{uqgln uqglm duality} proves that the irreducibles appearing in the decomposition of $\bigwedge_q(V^{(nm)})$ as a $\Uqgln \otimes \Uqglm$-module are parametrized by the same weights as the irreducibles in the decomposition of $\bigwedge(\mathbb{C}^{nm}) \cong \bigwedge(\mathbb{C}^n \otimes \mathbb{C}^m)$ as a $\mathfrak{gl}_n \otimes \mathfrak{gl}_m$-module. Moreover, our decomposition de-quantizes appropriately, which means that the endomorphisms induced by the quantum group generators tend to those induced by their classical counterparts in the limit $q \to 1$. For instance, compare the action of $E_i^{(n)}$ and $E_j^{(m)}$ on an arbitrary basis vector in the quantum case as computed in \Cref{uqgln ei on basis} and \Cref{uqglm ej on basis} to the action of their classical counterparts in the classical case, as computed in \Cref{gln ei on basis} and \Cref{glm ej on basis}.

Our proof of the first statement in \Cref{uqgln uqglm duality} mimics the argument establishing Theorem $6.16$ in \cite{lzz_2010}. First we construct a joint highest weight vector for each isotypic component in the decomposition of \Cref{uqgln uqglm duality decomp} in the next lemma, and then we use a dimension count and the classical skew $GL_n \times GL_m$ duality \Cref{skew gln glm duality} to conclude. 

Recall \Cref{hwv skew gln glm}, which computes highest weight vectors in the classical case. In the quantum case, the highest weight vectors are essentially the same: for each partition $\mu$ fitting in an $n \times m$ rectangle we simply take the product of all basis vectors in the boxes occupied by the Young diagram corresponding to $\mu$. \Cref{hwv young diagram} illustrates an example.

\begin{lem}\label[lem]{uqgln uqglm hwv}
	For each partition $\mu = (\mu_1, \ldots, \mu_n)$ satisfying $\mu_1 \leq m$, let
	$$
		v_\mu = 
		\left(v_1 \cdots v_{1 + (\mu_1 - 1)n} \right)
		\left(v_2 \cdots v_{2 + (\mu_2 - 1)n} \right)
		\cdots 
		\left(v_n \cdots v_{2 + (\mu_2 - 1)n}\right).
	$$ 
	Then
	\begin{enumerate}[(1)]
		\item The element $v_\mu$ is a highest weight vector with respect to the actions of $\Uqgln$ and of $\Uqglm$.
		\item The $\Uqgln$-weight of $v_\mu$ is $\mu$, and its $\Uqglm$-weight is $\mu'$.
	\end{enumerate}
\end{lem}

\begin{proof}
	It is understood that if $\mu_i = 0$, then $v_\mu$ contains no $v_{i + (b-1)n}$ as a factor. Since $\mu$ is a partition, this also implies that there is no $v_{j + (b-1)n}$ factor for every $j \geq i$.
	
	Much like in the classical case, \Cref{uqgln ei on basis,uqglm ej on basis} show that positive root vectors in $\Uqgln \otimes \Uqglm$ shift occupied positions in any given state vector above and to the left, as illustrated by \Cref{action of ein,action of ejm}. Since $\mu$ is a left-justified partition, every position above and to the left of an occupied position in $v_\mu$ is occupied already, so
	$$\lambda_q(E_i^{(n)}) \, v_\mu = \rho_q(E_j^{(m)}) \, v_\mu = 0,$$
	for $i = 1, \ldots, n-1$ and $j = 1, \ldots, m-1$. Hence, each $v_\mu$ is a joint $\Uqgln \otimes \Uqglm$-highest weight vector.

	We can readily prove $(2)$ using \Cref{uqgln clqnm embedding,uqglm clqnm embedding}. For instance, we see that
	\begin{align*}
		\lambda_q(K_{i}^{(n)}) \, v(\ell) &= q^{\sum_{j=1}^m \ell_{i + (j-1)n} - \ell_{i+1 + (j-1)n}} \, v(\ell), \quad \text{and} \\
		\rho_q(K_{j}^{(m)}) \, v(\ell) &= q^{\sum_{i=1}^n \ell_{i + (j-1)n} - \ell_{i + jn}} \, v(\ell),
	\end{align*}
	for any $v(\ell)$, so 
	\begin{align*}
		\lambda_q(K_{i}^{(n)}) \, v_\mu &= q^{\mu_i - \mu_{i+1}} \, v_\mu = q^{\langle \alpha_i^{(n)}, \, \mu \rangle} \, v_\mu, \quad \text{and} \\
		\rho_q(K_{j}^{(m)}) \, v_\mu &= q^{\mu_j' - \mu_{j+1}'} \, v_\mu = q^{\langle \alpha_j^{(m)}, \, \mu' \rangle} \, v_\mu.
	\end{align*}

	Here we used $\alpha_i^{(p)} = \epsilon_i - \epsilon_{i+1}$ to denote the $i$th simple positive root of $\mathfrak{gl}_p$.
\end{proof}

\begin{proof}[Proof of \Cref{uqgln uqglm duality}.]
	For any $\mathbb{Z}_{\geq 0}$-graded module $M$, let $M_{= k}$ denote the graded component of degree $k$. It follows from \Cref{uqgln uqglm hwv} that $\bigoplus_\mu V_\mu^{(n)} \otimes V_{\mu'}^{(m)}$ is a $\Uqgln \otimes \Uqglm$-submodule of $\bigwedge_q(V^{(nm)})$. Let $|\mu|$ denote the size of the partition $\mu$, that is, the sum of its parts. \Cref{uqgln ei on basis,uqglm ej on basis} prove that the $\Uqgln$- and the $\Uqglm$-actions preserve subspaces of homogeneous degree, so that in fact
	\begin{align*}
		\bigoplus_{|\mu| = k} V_\mu^{(n)} \otimes V_{\mu'}^{(m)} \subseteq \textstyle \left(\bigwedge_q(V^{(nm)})\right)_{=k}
	\end{align*}
	for each $k \leq nm$. The classical skew duality \Cref{skew gln glm duality} then implies that
	\begin{align*}
		\sum_{|\mu| = k} 
			\dim_{\mathbb{C}(q)}\bigg(V_\mu^{(n)} \otimes V_{\mu'}^{(m)}\bigg) 
		= \dim_{\mathbb{C}} \bigwedge(\mathbb{C}^n \otimes \mathbb{C}^m),
	\end{align*}
	since irreducible modules of $\mathfrak{gl}_p$ and of $U_q(\mathfrak{gl}_p)$ parametrized by the same dominant weight have the same dimension. The $\Uqglnm$-module $V^{(nm)}$ is flat in the sense of Lemma~2.32 in \cite{berenstein}, so the $v(\ell)$ form a basis of $\bigwedge_q(V^{(nm)})$, which implies that $\dim_{\mathbb{C}(q)}\bigwedge_q(V^{(nm)}) = \dim_{\mathbb{C}}\bigwedge(\mathbb{C}^n \otimes \mathbb{C}^m)$ and concludes our proof.
\end{proof}



\bibliographystyle{alphaurl}
\bibliography{ref}
\end{document}